\pdfoutput=1

\documentclass[final, reqno, 10pt]{amsart}

 \usepackage{amsmath,amsthm,amsfonts}
\usepackage{latexsym,mathabx}
\usepackage{amssymb}
\usepackage{bbold}
\usepackage{enumerate} 
\usepackage{color}
\usepackage{ulem}
\usepackage{tikz}
\usepackage{hyperref}

\newtheorem{theorem}{Theorem}[section]
\newtheorem{lemma}[theorem]{Lemma}
\newtheorem{prop}[theorem]{Proposition}
\newtheorem{corollary}[theorem]{Corollary}
\newtheorem{cor}[theorem]{Corollary}
\newtheorem{rmk}[theorem]{Remark}
\theoremstyle{remark}

\newtheorem{example}[theorem]{Example}

\newtheorem{definition}{Definition}[section]
\newcommand{\set}{\mathbb}

\newcommand{\les}{\lesssim}

\newcommand{\mc}{\mathcal}
\newcommand{\be}{\begin{equation}}
\newcommand{\ee}{\end{equation}}
\newcommand{\bee}{\begin{align}}
\newcommand{\eee}{\end{align}}
\newcommand{\ba}{\begin{array}}

\newcommand{\ea}{\end{array}}
\newcommand{\bpm}{\begin{pmatrix}}
\newcommand{\epm}{\end{pmatrix}}

\newcommand{\la}{\langle}
\newcommand{\ra}{\rangle}

\def\calL{\mc{L}}
\def\wtil{\widetilde}

\def\dist{\mathrm{dist}}

\DeclareMathOperator{\supp}{supp}

\DeclareMathOperator{\Imim}{Im}

\renewcommand{\Im}{\Imim}
\newcommand{\R}{\mathbb{R}}
\newcommand{\C}{\mathbb{C}}
\newcommand{\Z}{\mathbb{Z}}
\def\Sph{\mathbb{S}}
\newcommand{\N}{\mathbb{N}}

\newcommand{\one}{{\mathbb{1}}}
\newcommand{\EQ}[1]{\begin{equation}  \begin{split} #1 \end{split} \end{equation} }

\def\eps{\varepsilon}
\def\what{\widehat}

\def\calB{\mc{B}}

\def\calS{\mc{S}}

\def\nn{\nonumber}

\def\wtild{\widetilde}

\def\arcsn{\mathrm{arcsn}}
\def\sn{\mathrm{sn}}
\def\eps{\varepsilon}
\def\arcsinh{\mathrm{arcsinh}}
\def\D{\set{D}}
\def\Rect{\mathcal{R}}
\def\ol{\overline}
\def\del{\partial}
\def\Erw{\mathbb{E}}
\def\S{\set{S}}
\def\calE{\mathcal{E}}
\def\cC{\mathcal{C}}
\renewcommand{\em}{\it}
\def\uk{\underline{k}}
\def\uell{\underline{\ell}}

\numberwithin{equation}{section}

\begin{document}

\title[A higher dimensional Bourgain-Dyatlov FUP]{A higher dimensional Bourgain-Dyatlov fractal uncertainty principle}

\author{Rui Han}
\thanks{The authors thank Jean Bourgain for his interest in this paper, and thank Semyon Dyatlov for his comments on a preliminary version which improved our presentation.
We thank the IAS, Princeton, for its hospitality during the 2017-18 academic year.
This material is based upon work supported by the National Science Foundation under Grant No. DMS-1638352.}
\address{School of Mathematics, IAS, 1 Einstein Drive, Princeton, NJ 08540 }
\email{rhan@ias.edu}

\author{Wilhelm Schlag}
\thanks{The second author was partially supported by NSF grant DMS-1500696
during the preparation of this work.}
\address{Department of Mathematics, The University of Chicago, 5734 South University Avenue, 
Chicago, IL 60637}
\email{schlag@math.uchicago.edu}

\begin{abstract}
We establish a version of the fractal uncertainty principle, obtained by Bourgain and Dyatlov in 2016, in higher dimensions. 
The Fourier support is limited to sets $Y\subset \R^d$ which can be covered by finitely many products of $\delta$-regular sets in one dimension, but relative to arbitrary axes. 
Our results remain true if Y is distorted by diffeomorphisms.
Our method combines the original approach by Bourgain and Dyatlov, in the more quantitative 2017 rendition by Jin and Zhang, with Cartan set techniques. 
\end{abstract}

\maketitle

\section{Introduction}

Bourgain-Dyatlov~\cite{BD} proved the following result. 

\begin{theorem}\label{thm:BD16}
Let $X,Y\subset\R$ and $N\ge1$ be such that (i) $X\subset [-1,1]$ is $\delta$-regular with constant $C_R$ on scales $N^{-1}$ and $1$, (ii) $Y\subset[-N,N]$,  is $\delta$-regular with constant $C_R$ on scales $1$ and $N$. Then there exist constants $\beta>0$, and $C$ depending on $\delta, C_R$ so that 
\[
\|f\|_{L^2(X)} \le C N^{-\beta} \|f\|_{L^2(\R)}
\]
for all $f\in L^2(\R)$ with $\supp(\hat{f})\subset Y$. 
\end{theorem}

The $\delta$-regularity condition is akin to asking for a  Frostman measure at dimension~$\delta$, see Definition~\ref{def:deltareg}  below for the precise statement.  Theorem~\ref{thm:BD16} is most interesting for $\delta$ close to~$1$. For $\delta<\frac12$, Cauchy-Schwarz and measure estimates in phase space suffice. The $\beta$ was made effective later by Jin and Zhang~\cite{JZ}.  Combining this  fractal uncertainty principle with earlier results by Dyatlov and Zahl~\cite{DyZ} led to a breakthrough on the  existence for an essential spectral gap for convex co-compact hyperbolic surfaces. This refers to a strip to the left of the $1/2$ line in the complex plane in which the Selberg zeta function has only finitely many zeros. This result can be reformulated in terms of strips below the real axis in which the meromorphic continuation of the resolvent of the Laplacian of the hyperbolic surface exhibits only finitely many resonances. This in turn can be rephrased as a decay rate of the resolvent for large energies within such  a strip. 

For other applications see~\cite{BD2, DyJ, DyJ2}, and for a survey~\cite{DyaExp}. 

It remained an open problem to establish an analogue of Theorem~\ref{thm:BD16} in higher dimensions. This is main goal of this paper. Here we focus entirely on this problem in harmonic analysis, and plan to present applications elsewhere.

\medskip

We now present our main results. Let $X\subset [-1,1]^d$ be a $\delta$-regular set in the sense of Bourgain-Dyatlov with $\delta\in (0,d)$ and constant~$C_R$, on scales $N^{-1}$ to~$1$. 
In~\cite{BD} this concept is defined only on the line, but the definition, together with its main properties, carries over to higher dimensions. Strictly speaking, we do not need the regularity condition per se, but rather the porosity property of such sets as stated precisely in Definition~\ref{def:porous} below. Second, let $Y\subset [-N,N]^d$ be of the form 
\EQ{ \label{eq:Y intro}
{{ Y=\Big \{ \sum_{i=1}^d  \xi_{i} \vec e_{i} \:|\: \xi_{i}  \in Y_{i} \Big\}, }}
}
where $\vec e_{i}$ are unit vectors with $|\det(\vec e_{1},\ldots,\vec e_{d})|\ge \eps_0$, a positive constant (possibly small), 
and $Y_{i}\subset [-{{2}}N,{{2}}N]$ is a $\delta_1$-regular set with $\delta_1\in (0,1)$ and constant~$C_R$, on scales $1$ to~$N$.

\begin{theorem}
\label{thm:main1}
Let $X,Y$ be as in the previous paragraph in dimension $d\ge2$. 
Then there exists constant
$C=C(d, \eps_0, \delta, \delta_1, C_R)>0$ such that for 
{{
\[\beta= \exp\Big\{ -\exp\Big[ \Big(\frac{C_R^2}{\iota \delta_1(1-\delta_1)}\Big)^{\frac{6-2\delta}{(1-\delta_1)(2-\delta)}}\Big] \Big\},\]
where $\iota>0$ is a small constant depending on $d$ and $\varepsilon_0$,}} and for 
any $f\in L^2(\R^d)$ with $\supp(\what{f})\subset Y$ one has 
\EQ{\label{eq:BDd}
\|f\|_{L^2(X)} \le C N^{-\beta} \|f\|_{L^2(\R^d)}
}
for sufficiently large $N\ge N_0(d, \eps_0,\delta,\delta_1, C_R)$. 
\end{theorem}

{{As a corollary of our main theorem, we allow $Y$ to be covered by the union of a finite number of $Y_{j}$'s, each satisfying
\eqref{eq:Y intro} but with a uniform $\eps_0$.}}
\EQ{ \label{eq:YYj}
Y &\subset \bigcup_{j=1}^m Y_j \\
Y_j &= \Big \{ \sum_{i=1}^d  \xi_{j,i} \vec e_{j,i} \:|\: \xi_{j,i}  \in Y_{j,i} \Big\}.
}
{{Furthermore, the number $m$ of covers can grow in $N$.
To be specific, we prove}}
 
\begin{corollary}
\label{thm:main2}
Let $X$ be as above and {{$Y$ be as in \eqref{eq:YYj}.}} 
Suppose $m$ grows with $N$ as follows
\[
m=\lfloor N^{\gamma} \rfloor,
\]
in which $0\leq \gamma<\beta$.
Then for any $f\in L^2(\R^d)$ with $\supp(\what{f})\subset Y$, {{and constants $C, \beta$ in Theorem \ref{thm:main1}}}, one has 
\EQ{\label{eq:BDd2}
\|f\|_{L^2(X)} \le C N^{\gamma-\beta} \|f\|_2,
}
for sufficiently large $N\ge N_0(d, \eps_0,\delta, \delta_1, C_R)$. 
\end{corollary}

\smallskip

{{Theorem~\ref{thm:main1} and Corollary~\ref{thm:main2}}} require that the Fourier support $Y$ may be covered by products of regular sets in one dimension {\em along lines}, cf.~\eqref{eq:YYj}. Our third result asserts that one may distort these lines by means of diffeomorphisms which are obtained as follows. 
Let $\Psi_0:[-1,1]^d\to [-1,1]^d$ by a smooth map with a smooth inverse so that 
 \EQ{
 \label{eq:Psi0}
 \| D\Psi_0\|_\infty + \| D\Psi_0^{-1}\|_\infty + \| D^2\Psi_0\|_\infty  \le D_0. 
 }
where the supremum norm is taken over the cube.  For $N\ge1$ we consider the diffeomorphisms $\Phi_N(x) = N\Psi_0(x/N)$ from the cube 
$[-N,N]^d\to [-N,N]^d$.   Then 
\EQ{
 \label{eq:PhiN}
 \| D\Phi_N\|_\infty + \| D\Phi_N^{-1}\|_\infty + N\| D^2\Phi_N\|_\infty  \le C(d,D_0), 
 }
where the supremum norm is taken over the cube $[-N,N]^d$.   

\begin{theorem}
\label{thm:main3}
  Theorem~\ref{thm:main1} remains correct with   $\Phi_N(Y)$ in place of~$Y$.  
  Constants  depend on $D_0$, but not on $\Psi_0$. 
\end{theorem}

\medskip

In the following section we demonstrate the Cartan techniques by reproving a certain step in~\cite{BD} which was proved there by means of harmonic measure of the strip with a real line-segment removed. In Section~\ref{sec:locdim} we go beyond the one-dimensional setting via these Cartan methods. The subsequent sections implement the argument in analogy with~\cite{BD} albeit in dimensions two and higher. We haven striven to present the argument in a modular fashion. In particular, the delicate Beurling-Malliavin step appears only in Section~\ref{sec:Y} in order to prove the existence of {\em damping functions}. We do not use a higher-dimensional version of the Beurling-Malliavin theorem which appears to be unknown. Rather, we reduce ourselves in that step to the aforementioned product structure of~$Y$ (or covers of finitely many of such products) precisely so as to be able to still use the one-dimensional construction of such damping functions. Moreover, as in~\cite{JZ} it is important for us to use the weaker form of the Beurling-Malliavin theorem obtained via outer functions, see~\cite{seven}. 
Any other construction of damping functions in Section~\ref{sec:Y} would lead to different formulations of our main theorems in terms of the conditions on~$Y$ without needing to change anything in  the other sections.    Theorem~\ref{thm:main3} is proved in Section~\ref{sec:distort}. 

\section{$L^2$ localization in one dimension}

Throughout, we let  $\Rect(q)$ be the rectangle with vertices $\pm iq$, $1\pm iq$.  We begin with quantitative bounds on the Schwarz-Christoffel map from the disk onto a rectangle. The goal is to control this conformal mapping as the eccentricity of $\Rect(q)$ tends to~$0$. 

\begin{lemma}
\label{lem:Phir}
Let $0<q \le 1$ and 
define  $\Phi_{q}$ to be the unique conformal map, continuous up to the boundary,  which takes the unit disk $\D$ onto the rectangle $\Rect(q)$  and so that $\Phi_{q}( -1)=0$, $\Phi_{q}(\pm i) = \pm i q$. Then $\Phi_{q}(1)=1$,  $\Phi_{q}(e^{\pm i\theta(q)} )= 1\pm iq$ where 
\[
\theta(q) = 8 \exp\big(-\frac{\pi}{2q} \big) (1+O(q)),\quad q\to0
\]
Moreover, 
\[
\Phi_{q}([a_{1}(q), a_{2}(q)]) = \big[ \frac14, \frac34\big], \quad a_{j}(q) =1 -\delta_{j}(q)
\]
with 
\[
\delta_{1}(q) = 4 \exp\big(-\frac{\pi}{8q}\big) (1+O(q)), \quad \delta_{2}(q) = 4 \exp\big(-\frac{3\pi}{8q}\big) (1+O(q))
\]
as $q\to0$.  Let $E\subset [a_{1}(q), a_{2}(q)]$ be a measurable set. Then for sufficiently small $q$ one has $| \Phi_{q}(E)|\le 2\delta_{2}(q)^{-2}|E|$, where $|\cdot|$ denotes Lebesgue measure. 
\end{lemma}
\begin{figure*}[t]
\begin{tikzpicture}[yscale=0.8,xscale=0.8]
\draw [fill=gray!35!] (2,2) circle (2);
\draw [->,line width = .03cm] (5,2) -- (7,2);
\draw [-] (10,1.5) -- (10,2.5);
\draw [-] (10,2.5) -- (12,2.5);
\draw [-] (10,1.5) -- (12,1.5);
\draw [-] (12,1.5) -- (12,2.5);
\draw [-] (10.5,2) -- (11.5,2);
\draw [fill=gray!35!] (10,1.5) rectangle (12,2.5);
\draw [-] (3,2) -- (3.8,2);
\node at (6,2.4) {$\Phi_q$};
\node at (8.95,2.7) {\small $\Phi_q(i)=iq$};
\node at (8.8,2) {\small $\Phi_q(-1)=0$};
\node at (8.6,1.3) {\small $\Phi_q(-i)=-iq$};
\node at (14.1,1.3) {\small $\Phi_q(e^{-i\theta(q)})=1-iq$};
\node at (13.2,2) {\small $\Phi_q(1)=1$};
\node at (14,2.7) {\small $\Phi_q(e^{i\theta(q)})=1+iq$};
\node at (4.60211,2.61803) {$e^{i \theta(q)}$};
\node at (4.60211,1.382) {$e^{-i \theta(q)}$};
\node at (3,2.2) {$a_1$};
\node at (3.7,2.2) {$a_2$};
\node at (10.5,2.2) {\small $1/4$};
\node at (11.5,2.2) {\small $3/4$};
\node at (2,-0.5) {$\mathbb{D}$};
\node at (11, 0) {$\mathcal{R}(q)$};
\filldraw[color=black, fill=black](10.5,2) circle (.05);
\filldraw[color=black, fill=black](11.5,2) circle (.05);
\filldraw[color=black, fill=black](3,2) circle (.05);
\filldraw[color=black, fill=black](3.8,2) circle (.05);
\filldraw[color=black, fill=black](3.90211,2.61803) circle (.05);
\filldraw[color=black, fill=black](3.90211,1.382) circle (.05);
\filldraw[color=black, fill=black](12,1.5) circle (.05);
\filldraw[color=black, fill=black](12,2.5) circle (.05);
\filldraw[color=black, fill=black](10,1.5) circle (.05);
\filldraw[color=black, fill=black](10,2.5) circle (.05);
\filldraw[color=black, fill=black](12,2) circle (.05);
\filldraw[color=black, fill=black](10,2) circle (.05);
\end{tikzpicture}
\caption{Conformal map $\Phi_q$}
\end{figure*}

\begin{proof}
Let $0<k<1$ and consider the elliptic integral of the first kind
\[
\arcsn(z,k) = \int_{0}^{z} \frac{dt}{\sqrt{(1-t^{2})(1-k^{2}t^{2})}}, \qquad \Im z>0
\]
which maps the upper half-plane onto the rectangle with vertices $\pm L(k), \pm L(k) + iH(k)$. Here $2L(k)$ and $iH(k)$ are the periods of the elliptic function $\sn(z,k)$ and satisfy, as $k\to0$, 
\EQ{\nn
L(k) &=  \int_{0}^{1}\frac{dt}{\sqrt{(1-t^{2})(1-k^{2}t^{2})}} = \frac{\pi}{2}+O(k^{2}), \\
H(k) &= \int_{1}^{k^{-1}}   \frac{dt}{\sqrt{(t^{2}-1)(1-k^{2}t^{2})}}   = \int_{0}^{\infty} \frac{ds}{\sqrt{(1+s^{2})(1+k^{2}s^{2})}}  \\
&= \log 4 -\log k + O(k)
}
The latter expansion is a standard fact, see for example \cite[Section 17.3.26]{AS}. 
Let $q:=\frac{L(k)}{H(k)}$ and set
\EQ{\label{eq:Fq}
F_{q}(z) = - \frac{i}{H(k)} \arcsn(z,k)
}
which maps the upper half-plane onto the rectangle with vertices $\pm iq, 1\pm iq$.    With $k=e^{-\frac{\pi}{2}\ell}$, 
\EQ{ \nn
q &= \frac{\frac{\pi}{2}+O(k^{2})}{ \log 4 + \frac{\pi}{2}\ell  + O(k)} \\
&= \ell^{-1}\big(1  - \frac{\log{16}}{\pi \ell} + O(k)\big)
}
and thus
\EQ{\nn 
\ell &= q^{-1}\big(1  - \frac{2\log{4}}{\pi }q + O(q^{2})\big) \\
k &= 4 \exp\big(-\frac{\pi}{2q} \big) (1+O(q))
}
Define $A(q), B(q)$ by $F_{q}(iA(q))=\frac14$, $F_{q}(iB(q))=\frac34$.  Thus, 
\EQ{\nn 
\int_{0}^{A(q)} \frac{ds}{\sqrt{(1+s^{2})(1+k^{2}s^{2})}} &= \frac14 H(k)\\
\int_{0}^{B(q)} \frac{ds}{\sqrt{(1+s^{2})(1+k^{2}s^{2})}} &= \frac34 H(k)
}
\begin{figure*}[t]
\begin{tikzpicture}[yscale=0.8,xscale=0.8]
\draw [-] (-1,-0.3) -- (4,-0.3);
\draw [->,line width = .03cm] (4.4,0.8) -- (7.6,0.8);
\node at (6,1.2) {$g_k(\cdot):=\mathrm{arcsn}(\cdot,k)$};
\node at (1.5,-1.5) {Upper half plane};
\fill[gray!35!] (-1,-0.3) rectangle (4,3.3);
\fill[gray!35!] (9.5,-1) rectangle (11.5,3);
\draw [-] (9.5,-1) -- (9.5,3);
\draw [-] (9.5,-1) -- (11.5,-1);
\draw [-] (11.5,-1) -- (11.5,3);
\draw [-] (9.5,3) -- (11.5,3);
\draw [-] (1.5,1.5) -- (1.5,2);
\draw [-] (10.5, 0) -- (10.5,2);
\node at (8,-0.8) {\small $g_k(-1)=-L(k)$};
\node at (10.5, -1.5) {\small $g_k(0)=0$};
\node at (12.8, -0.8) {\small $g_k(1)=L(k)$};
\node at (7.85, 3) {\small $g_k(-k^{-1})=$};
\node at (8.1, 2.5) {\small $-L(k)+i H(k)$};
\node at (12.8, 3) {\small $g_k(k^{-1})=$};
\node at (13, 2.5) {\small $L(k)+i H(k)$};
\node at (10.5, 3.3) {\small $g_k(\infty)$};
\node at (-0.5,-0.7) {$-k^{-1}$};
\node at (1, -0.7) {$-1$};
\node at (1.5, -0.7) {$0$};
\node at (2, -0.7) {$1$};
\node at (3.5, -0.7) {$k^{-1}$};
\node at (2.2, 1.3) {$iA(q)$};
\node at (2.2, 2.2) {$iB(q)$};
\node at (11, -0.3) {$\frac{i H(k)}{4}$};
\node at (10.8, 2.3) {$\frac{3i H(k)}{4}$};
\filldraw[color=black, fill=black](10.5,2) circle (.05);
\filldraw[color=black, fill=black](10.5,0) circle (.05);
\filldraw[color=black, fill=black](1.5,2) circle (.05);
\filldraw[color=black, fill=black](1.5,1.5) circle (.05);
\filldraw[color=black, fill=black](9.5,-1) circle (.05);
\filldraw[color=black, fill=black](9.5,3) circle (.05);
\filldraw[color=black, fill=black](11.5,-1) circle (.05);
\filldraw[color=black, fill=black](11.5,3) circle (.05);
\filldraw[color=black, fill=black](10.5,-1) circle (.05);
\filldraw[color=black, fill=black](10.5,3) circle (.05);
\filldraw[color=black, fill=black](-0.5,-0.3) circle (.05);
\filldraw[color=black, fill=black](1,-0.3) circle (.05);
\filldraw[color=black, fill=black](1.5,-0.3) circle (.05);
\filldraw[color=black, fill=black](2,-0.3) circle (.05);
\filldraw[color=black, fill=black](3.5,-0.3) circle (.05);
\end{tikzpicture}
\caption{Elliptic integral $\mathrm{arcsn}(z,k)$}
\end{figure*}
We make the ansatz $A(q)= c k^{-\frac14}(1+\eps(q))$.    Then 
\EQ{\nn 
\int_{0}^{A(q)} \frac{ds}{\sqrt{(1+s^{2})(1+k^{2}s^{2})}}  & = (1+O(k^{\frac32}))  \int_{0}^{A(q)} \frac{ds}{\sqrt{ 1+s^{2}  } }   \\
&= \arcsinh( c k^{-\frac14}(1+\eps(q))    ( 1+O(k^{\frac32}))  \\
&= \log(2c k^{-\frac14}(1+\eps(q)) ( 1+O(k^{\frac32})) \\
&= \frac14 (\log 4 -\log k + O(k)) 
}
Hence, 
\EQ{\nn
\log(2c) -\frac14\log  k + \log(1+\eps(q))  & = \frac14 (\log 4 -\log k + O(k))   \\
c & = \frac12\sqrt{2}, \qquad  \eps(q) = O(k) \\
A(q) &= \frac12\sqrt{2} \, k^{-\frac14}(1+O(k)) 
}
Similarly,  with $B(q) = \tilde c k^{-\frac34}(1+\tilde \eps(q))$
\EQ{\nn
  \log(2\tilde c) -  \frac34 \log k + \log(1+\tilde\eps(q)) 
&= \frac34 (\log 4 -\log k + O(k))(1+ O(k^{\frac12}))   \\
\tilde c = \sqrt{2}, \quad \tilde \eps(q) = O(k^{\frac12}\log k)
}
and so 
\[
B(q) = \sqrt{2} \, k^{-\frac34}(1+O(k^{\frac12}\log k)) 
\]
Expressing $k$ in terms of $q$ we obtain
\[
A(q) = \frac12 \exp\big(\frac{\pi}{8q}\big) (1+O(q)), \quad B(q) = \frac12 \exp\big(\frac{3\pi}{8q}\big) (1+O(q))
\]
Next, we conformally map the upper half plane $\Im z>0$ onto the unit disk $|w|<1$ via 
 $z=\varphi(w) = i \frac{w+1}{1-w}$, $w=\frac{z-i}{z+i}$.   One has $\varphi(-1)=0$, $\varphi(\pm i) = \mp 1$,  $\varphi(e^{i\theta}) = -k^{-1}$ with $\theta= 2k +O(k^{3})$. 
 Furthermore,  $\varphi([a_{1}(q), a_{2}(q)]) = i[A(q), B(q)]$ where 
\EQ{\nn
a_{1}(q) &= \frac{A(q)-1}{A(q)+1} = 1 - 2 A(q)^{-1} + O(A(q)^{-2}),\\ a_{2}(q) &= \frac{B(q)-1}{B(q)+1} = 1 - 2 B(q)^{-1} + O(B(q)^{-2})
}
Setting $a_{j}(q) = 1 -\delta_{j}(q)$ we have 
\[
\delta_{1}(q) = 4 \exp\big(-\frac{\pi}{8q}\big) (1+O(q)), \quad \delta_{2}(q) = 4 \exp\big(-\frac{3\pi}{8q}\big) (1+O(q))
\]
as claimed.  The final claim of the lemma follows from $$|(F_{q}\circ\varphi)'(w)|\le |F_{q}'(z)||\varphi'(w)|\le 2(1-|w|)^{-2}$$ where $\varphi(w)=z$, $w\in(0,1)$. We used here that 
for $z=is$, $s>0$, 
\EQ{\nn 
|F_{q}'(z)| &= H(k)^{-1} (1+|z|^{2})^{-\frac12}(1+k^{2}|z|^{2})^{-\frac12} \\
&\le H(k)^{-1} (1+|z|^{2})^{-\frac12} \le 1
}
for small $q$. 
 \end{proof}

By a subharmonic function $v$ on a domain $\Omega\subset\C$ we mean a function $v:\Omega\to [-\infty,\infty)$, which is upper semi-continuous and satisfies the sub mean-value property.  We recall the basic Riesz representation of subharmonic function on the disk, albeit with precise quantitative control on the Riesz mass and the harmonic part. In view of Lemma~\ref{lem:Phir} we need  to consider the case where the lower bound on the subharmonic function is attained arbitrarily  close to the boundary of the unit disk. 

\begin{lemma}\label{lem:vRiesz}
Let $v$ be subharmonic on a neighborhood of $\D$, with $v\le M$ on $\D$, and assume $\sup_{\rho \D} v\ge m$ for some $0<\rho<1$.   
Let $\rho<r_{1}<r<1$. 
Then there exist a nonnegative measure $\mu$ on $\D$, called the Riesz measure,  with the property that for all $w\in r\D$ 
\EQ{\label{eq:vRiesz}
v(w)=\int_{r\D}\log{|z-w|}\, \mu(d z)+h(w) 
}
with $h$ harmonic on $r\D$. We have the quantitative bounds on the Riesz mass
\EQ{\label{eq:mu}
\mu(r \D)\leq \frac{M-m}{\log\Big(\frac{1+\rho r}{\rho + r}\Big)}
}
and on the deviations of the harmonic function 
\EQ{\label{eq:hminc}
\min_{c\in \R}\max_{|w|\leq r_{1}}\ |h(w)-c|  \le  \frac12 (M-m)  \frac{r+r_{1}}{r-r_{1}} \frac{ \log\Big(\frac{1+\rho r}{1-r^{2}}\Big)  }{\log\Big(\frac{1+\rho r}{\rho + r}\Big)  }
=:\varepsilon
}
The constant $c$ which minimizes the left-hand side satisfies
\EQ{\label{eq:cm}
c \ge m  - \varepsilon - \log(r+\rho) \mu(r \D) 
}
\end{lemma}
\begin{proof}
We will assume that $v$ is smooth, the general case following by approximation. 
The Green function $G:\D\times \D \to \R$ given by
$$
G(z,w):=\frac{1}{2\pi }\log\Big |\frac{z-w}{1-z\ol{w}}\Big |
$$
satisfies $\Delta_z G(z,w)=\delta_w$ and $G(z,w)=0$ when $|z|=1$.

Let $w\in \D$. By Green's second identity for the domain $\D$, we have
$$
v(w)-\int_{\D}G(z,w) \Delta v(z)\, \mathrm{Vol}(d z)
= \int_{\del \D} v(z) \frac{\del G}{\del n_z}(z,w) \, \sigma(d z), 
$$
where $\mathrm{Vol}$ is the standard volume measure and $\sigma$ is the (unnormalized) arc length measure on the circle $\del \D$. 
Since $v$ is smooth and subharmonic, $\Delta v$ is a non-negative, continuous function, call it $2\pi \mu$. Therefore
\begin{align}\label{eq:green}
v(w)=\int_{\D}2\pi G(z,w)\, \mu(d z)+h_0(w),
\end{align}
where
\begin{align}
h_0(w):=\int_{\del \D} v(z) \frac{\del G}{\del n_z}(z,w) \, \sigma(d z).
\end{align}
Let $0<r<1$. 
On the disk $r \D$ we  have the Riesz representation 
\EQ{\label{eq:riesz}
v(w)=\int_{r\D}\log{|z-w|}\, \mu(d z)+h(w),
}
where
\EQ{
\label{eq:hdefn}
h(w):=\int_{\D\setminus r\D} \log{\Big |\frac{z-w}{1-z\overline{w}}\Big |}\, \mu(d z)-\int_{r\D} \log |1-z\ol{w} |\, \mu(d z)+h_0(w)
}
is harmonic in $r\D$.
Note that  $\frac{\del G}{\del n_z}(z,w)$ is  the Poisson kernel whence 
\EQ{\label{eq:h0=poisson}
h_0(w)=\int_{0}^1 v(e(\theta)) P_{|w|}(\varphi-\theta)\, d\theta,\qquad w=|w|e(\varphi).
}
We now set out to bound the Riesz measure $\mu$.
Without loss of generality, assume $m=v(\rho)$. Then 
setting $w=\rho$ in \eqref{eq:green} yields 
\EQ{\label{eq:h0v}
 \int_{\D} \log \frac{|1-\rho z| }{ |z-\rho|}   \, \mu(d z)=h_0(\rho)-v(\rho)\leq M-m,
}
in which we used 
\EQ{\label{eq:h0M}
h_0(\rho)\leq M. 
}
This follows from the maximum principle and the fact that $h_0$ is the harmonic function on $\D$ with boundary values $v$ by \eqref{eq:h0=poisson}.
By an elementary calculation, 
\[
\min_{|z|\le \tilde\rho} \frac{|1-\rho z| }{ |z-\rho|} = \frac{1+\rho \tilde\rho}{\rho + \tilde\rho} > 1
\]
for all $0<\rho, \tilde\rho<1$.  Inserting this bound into \eqref{eq:h0v} implies that
\EQ{\label{eq:Rieszmass}
\mu(\tilde\rho\, \D)\leq \frac{M-m}{\log\Big(\frac{1+\rho \tilde\rho}{\rho + \tilde\rho}\Big)}.
}
In particular, this bound holds with $\tilde \rho=r$ as above. Let $\rho<r_{1}<r<1$. For all $w\in r\D$ we have 
\EQ{\label{eq:hoben*}
h(w) & = \int_{\D\setminus r\D} {  {2\pi}} G(w,z)\, \mu(dz) - \int_{r\D} \log|1-z\bar{w}|\, \mu(dz) + h_{0}(w)\\
&\leq - \log(1-r^{2})\mu(r\D) + M =: h^{*}
}
By Harnack's inequality on $r_{1}\D$ we conclude from this that  for any $w\in r_{1}\D$,
\[
(h^{*} - h(w)) \le \frac{r+r_{1}}{r-r_{1}} (h^{*}-h(\rho))
\]
whence
\[
h(w) \ge \frac{r+r_{1}}{r-r_{1}} h(\rho)  - \frac{2r_{1}}{r-r_{1}} h^{*}
\]
By \eqref{eq:riesz}, 
\EQ{\label{eq:hrho}
h(\rho) &= v(\rho) - \int_{r\D}\log|z-\rho|\, \mu(dz) \ge m - \log(r+\rho)\mu(r\D)
}
and thus
\[
h(w) \ge \frac{r+r_{1}}{r-r_{1}} \big( m - \log(r+\rho)\mu(r\D) \big)  - \frac{2r_{1}}{r-r_{1}} h^{*} =: h_{*}
\]
In summary, 
\EQ{\label{eq:hdev}
\min_{c\in \R}\max_{|w|\leq r_{1}}\ |h(w)-c| &\leq \frac12 (h^{*}-h_{*}) \\
&= \frac12  \frac{r+r_{1}}{r-r_{1}} \Big(  h^{*} - m +  \log(r+\rho)\mu(r\D)   \Big)  \\
&= \frac12  \frac{r+r_{1}}{r-r_{1}} \Big(M-m + \log\big( \frac{r+\rho}{1-r^{2}}  \big) \mu(r\D)   \Big)
}
Finally, bounded the $\mu$-mass by \eqref{eq:Rieszmass} finally implies that
\[
\min_{c\in \R}\max_{|w|\leq r_{1}}\ |h(w)-c|  \le  \frac12 (M-m)  \frac{r+r_{1}}{r-r_{1}} \frac{ \log\Big(\frac{1+\rho r}{1-r^{2}}\Big)  }{\log\Big(\frac{1+\rho r}{\rho + r}\Big)  }=:\varepsilon
\]
as claimed.  Finally, to establish~\eqref{eq:cm}, we return to \eqref{eq:hrho} and note that the left-hand side at most $c+\varepsilon$ for  $c$ the minimizer in the previous line. Then 
\[
c\ge m - \log(r+\rho)\mu(r\D)  - \varepsilon
\]
Note that one may insert \eqref{eq:Rieszmass} on the right-hand side to control the mass. 
\end{proof}

 We now apply the Cartan estimate for logarithmic potentials  to  the Riesz representation~\eqref{eq:vRiesz} in order to derive  lower bounds on $v$ up to a small measure of exceptions. 
 
 \begin{cor}
 \label{cor:vRiesz}
 Let $v$ be as in Lemma~\ref{lem:vRiesz} with $\rho=1-3\delta$,  $0<\delta< \frac13$. Then for all $0<H\le 1$
 there exist disks $D(z_{j},s_{j})$ so that 
 \[
v(z)\ge m  - (M-m) \big[2  \delta^{-3}\log(2/\delta) +  \delta^{-2} \log(2e/H) \big] 
 \]
 for all $z\in r_{1}\D \setminus  \bigcup_{j} D(z_{j},s_{j}) $  with $\sum_{j} s_{j} \le 5 H $ and $r_{1}=1-2\delta$. 
 \end{cor}
\begin{proof}
By Cartan's estimate for any $H>0$ there exist disks $D(z_{j},s_{j})$ such that  $\sum_{j} s_{j}\le {  {5}} H$ and 
\EQ{\label{eq:cartan}
\int_{r\D}\log{|w-z|}\, \mu(d w)  &\ge \mu(r\D) \log(H/e), \qquad \forall \, z\in r_{1}\D \setminus  \bigcup_{j} D(z_{j},s_{j}) 
}
See \cite{Lev}, Theorem~3, Section~11.2. 
To invoke the measure bound \eqref{eq:mu} we estimate
\EQ{\nn
\log\Big(\frac{1+\rho r}{\rho + r}\Big) & = \log \Big( \frac{2-4\delta+3\delta^{2}}{2-4\delta} \Big) \\
&{  {= \log\Big( 1 + \frac{3\delta^{2}}{2-4\delta} \Big) \ge \log(1+ \frac{3}{2}\delta^{2}) \ge \delta^{2}}}
}
since $\delta^{2}\le \frac12$ and $\log{(1+\frac{3}{2}x)}\geq x$ for $0\leq x\leq \frac{1}{2}$.   Consequently, 
\[
\mu(r\D) \le \delta^{-2}(M-m)
\]
 Next,
\EQ{\nn 
\frac{1+\rho r}{1-r^{2}} &\le \frac{2}{2\delta-\delta^{2}}\le \delta^{-1} (1+\delta)
}
as well as 
\[
\frac{r+r_{1}}{r-r_{1}} = \frac{2-3\delta}{\delta} \le 2\delta^{-1}
\]
whence \eqref{eq:hminc} implies 
\[
\min_{c\in \R}\max_{|w|\leq r_{1}}\ |h(w)-c| \le \varepsilon  \le (M-m) \delta^{-3}\log(2/\delta) =:\tilde{\varepsilon}.
\]
Finally, by \eqref{eq:cm}, one has 
\EQ{\nn
c &\ge m  - \varepsilon - \log(r+\rho) \mu(r\D)\ge m  - \varepsilon - \log(2) \mu(r\D) 
}
In view of \eqref{eq:vRiesz} and the preceding estimates we obtain
\EQ{
v(z) &\ge  c + \mu(r\D) \log(H/e)  - \varepsilon \ge m - 2 \varepsilon + \log(H/(2e))  \mu(r\D)\\
&\ge m  - (M-m) \big[2  \delta^{-3}\log(2/\delta) -  \delta^{-2} \log(H/(2e)) \big] 
}
for all $z$ as in \eqref{eq:cartan}. 
\end{proof}

By means of  the conformal transformation $\Phi_{q}$ from Lemma~\ref{lem:Phir} we can obtain a version of the Riesz representation 
theorem on thin rectangles $\Rect(q)$. 

\begin{cor}
\label{cor:RRiesz}
There exists $q_{*}\in (0,1]$ with the following property: 
let $u$ be subharmonic on $\Rect(q)$ for some $0<q\le q_{*}$, continuous up to the boundary. Assume that $u\le M$ on $\Rect(q)$ and $\max\limits_{x\in[1/4,3/4]}u(x)\ge m$.  Then 
\EQ{\label{eq:M-m}
u(x)\ge m  - (M-m)  \exp\big(\frac{9\pi}{8q}\big)  \big[\log(4) + \frac{9\pi}{4q}  +  \exp\big(-\frac{3\pi}{8q}\big)  \log(2e/H) \big] 
}
for all $x\in [1/4,3/4]\setminus \bigcup_{j}I_{j}$ where $\sum_{j} |I_{j}|\le  3  H \exp\big(\frac{3\pi}{4q}\big)$. 
\end{cor}
\begin{proof}
Let $v=u\circ \Phi_{q}$, with $\Phi_{q}$ as in Lemma~\ref{lem:Phir}. Then $v$ satisfies the assumptions of Corollary~\ref{cor:vRiesz} with $\rho\ge 1-\delta_{2}(q)$, and 
\EQ{\label{eq:deldel2}
\delta_{2}(q) & =4 \exp\big(-\frac{3\pi}{8q}\big) (1+O(q))\ge 3\delta \\
\delta &:= \exp\big(-\frac{3\pi}{8q}\big) {  {< \frac 13}},
}
provided $q_{*}$ is small enough. By Corollary~\ref{cor:vRiesz} we have 
\EQ{\nn 
v(z)& \ge m  - (M-m)  \exp\big(\frac{9\pi}{8q}\big)  \big[2 \log(2/\delta) +  \delta \log(2e/H) \big] \\
&= m  - (M-m)  \exp\big(\frac{9\pi}{8q}\big)  \big[\log(4) + \frac{9\pi}{4q}  +  \exp\big(-\frac{3\pi}{8q}\big)  \log(2e/H) \big] 
}
for all $z\in r_{1}\D\setminus \bigcup_{j} D(z_{j},s_{j})$, $\sum_{j} s_{j}\le {  {5}}H$, where $r_{1}=1-2\delta$.  The inverse image of $[1/4,3/4]$ under $\Phi_{q}$ is $[a_{1}(q),a_{2}(q)]$.  Define $\tilde I_{j}:= \R\bigcap D(z_{j},s_{j})$, $I_{j}=\Phi_{q}(\tilde I_{j})$,  and $E:=\bigcup_{j} \tilde I_{j}$ so that $\sum_{j}|\tilde I_{j}|\le  10  H$.  By Lemma~\ref{lem:Phir} we have 
\[
|\Phi_{q}(E)|\le {  {20}} H \delta_{2}(q)^{-2} {  {<3}} H \exp\big(\frac{3\pi}{4q}\big)
\]
as claimed.   
\end{proof}

Next, we apply the previous  results on subharmonic functions to $\log|F|$, where $F$ is analytic.

\begin{cor}\label{cor:muehe}
Let $F$ be an analytic function on a neighborhood of $\Rect(q)$ with  $0< q\leq q^*$, and $F$ not identically equal to zero.     
Denote 
$$ B_{1}:=\|F\|_{L^{2}([1/4,3/4])} ,\qquad   B_{2}:=\|F\|_{L^{2}(\del \Rect(q))} .$$
Then for some absolute constant $C_{0}$,  and all $H>0$, 
\EQ{
\label{eq:|F|}
   B_{1}^{K+1} &\le  e^{\frac{C_{0}K}{q}}\, B_{2}^{K} \, |F(x)|, \\
\text{holds for any}\ \ K &{  {\ge}} \exp\big(\frac{9\pi}{8q}\big)  \big[\log(4) + \frac{9\pi}{4q}  +  \exp\big(-\frac{3\pi}{8q}\big)  \log(2e/H) \big]  
}
for all $x\in [1/4,3/4]\setminus \bigcup_{j}I_{j}$ where $\sum_{j} |I_{j}|\le  3  H \exp\big(\frac{3\pi}{4q}\big)$. 
\end{cor}
\begin{proof}
We apply our previous results to  $u(z):=\log|F(z)|$, which is subharmonic on a neighborhood of $\Rect(q)$. However,  Corollary~\ref{cor:RRiesz} does not apply directly since we do not have a point wise upper bound on $u$. Returning to the subharmonic function $v=u\circ \Phi_{q}$ on the unit disk~$\D$, we note that the point wise upper bound $M$ on $v$ only entered through the estimate $h_{0}\le M$, see~\eqref{eq:h0M}, \eqref{eq:hoben*}.  The analytic function $\tilde F=F\circ \Phi_{q}$ satisfies $\log|\tilde F|=v$. Denoting by 
\EQ{\nn
P_{w}(d\theta)= P_{|w|}(d(\theta-\phi)){  {=\frac{1-|w|^2}{1-2|w| \cos{(2\pi(\theta-\phi))}+|w|^2} }}
} 
the Poisson kernel centered at $w=|w|e(\phi)$, we estimate $h_{0}$ from~\eqref{eq:h0v} as follows: 
\EQ{\label{eq:h0w2}
h_{0}(w) &= {  {\int_0^1}} v(e(\theta)) \, P_{w}(d\theta) 
= {  {\int_0^1}} \log|\tilde F(e(\theta)) |\, P_{w}(d\theta) \\
&\le  \log\Big ( {  {\int_0^1}} |\tilde F(e(\theta)) | \, P_{w}(d\theta)   \Big)  \\
&\le  \log\Big ( {  {\int_0^1}} |\tilde F(e(\theta)) | \, d\theta \; \big\| \frac{P_{w}(d\theta) }{d\theta}\big\|_{\infty}\Big) \\
&\le  \log  ( B_{2} ) + \log\big( \| \frac{d\theta}{d\sigma}\|_{L^{2}(\del\Rect(q))}     \big) + \log\big\| \frac{P_{w}(d\theta) }{d\theta}\big\|_{\infty}
}
where $d\sigma$ denotes arc length measure on $\del\Rect(q)$, and the correspondence between $\del\D$ and $\del\Rect(q)$ is given by $\xi\mapsto \Phi_{q}(e(\xi))$. 
On the one hand, 
\[
\big\| \frac{P_{w}(d\theta) }{d\theta}\big\|_{\infty} \le {  {2}} (1 -|w|)^{-1}
\]
and on the other hand, 
\EQ{\label{eq:bloed}
 \| \frac{d\theta}{d\sigma}\|_{L^{2}(\del\Rect(q))}^{2}  &=  \int_{\del\Rect(q)} \Big| \frac{d\theta}{d\sigma}\Big|^{2}\, d\sigma  = 
 {  {\int_0^1}} \Big| \frac{d\sigma}{d\theta}(\xi)\Big|^{-1}\, d\xi
}
Using the notations of  Lemma~\ref{lem:Phir}, the  boundary map $\del\D\to \del\Rect(q)$ induced by $\Phi_{q}$ is 
\EQ{\nn
\xi \mapsto \zeta(\xi) & := iH(k)^{-1}\arcsn(x(\xi),k),\\
 x(\xi) &:=\varphi(e(\xi)) = {  {- \cot(\pi\xi),\;\; x'(\xi)=\pi(1+x(\xi)^{2}) }}
}
where  $\varphi(w)= i\frac{w+1}{1-w}$  takes the disk to the upper half-plane. If $0< 2\pi \xi < \theta(q)$, then $\zeta(\xi)=1+iy(\xi)$ where 
\EQ{\nn
\frac{dy}{d\xi}= 
 \frac{\pi}{H(k)} \frac{1+x^2}{\sqrt{(x^2-1)(k^2 x^2-1)}}\geq \frac{\pi}{k H(k)},\quad x(\xi)<-k^{-1}. 
}
Therefore, this region contributes 
\[
\le  \frac{1}{2} k H(k) \theta(q) \les 1 \text{\ \ uniformly in \ }q
\]
to the integral in \eqref{eq:bloed}. Next, if $\theta(q) <  2\pi \xi <  \pi/2$, then $\zeta=u+iq$ with 
\EQ{\nn
\Big| \frac{du}{d\xi} \Big|
= \frac{\pi}{H(k)} \frac{1+x^2}{\sqrt{(x^2-1)(1-k^2 x^2)}}\geq \frac{\pi}{H(k)},\quad -k^{-1}<x(\xi)<-1  
}
and so this case contributes $\les H(k)$ to~\eqref{eq:bloed}. Finally, the region $\pi/2< 2\pi \xi <   2  \pi$ similarly adds at most $\les H(k)$ to~\eqref{eq:bloed}.

Combining these estimates with~\eqref{eq:h0w2}  yields
\EQ{\label{eq:Mnew}
h_{0}(w) &\le \log  ( B_{2} ) + \log\big(CH(k)) + \log(2/(\pi(1-r)) \\
&\le \log(B_{2}) + C_0 q^{-1} =: M
}
for all $|w|<r=1-\delta$ with some absolute constant $C_0$, cf.~\eqref{eq:deldel2}. This bound replaces \eqref{eq:h0M} and \eqref{eq:hoben*} above. 

As for the lower bound $m$ on $u$, one has $m\ge \log(B_{1})$ and thus \eqref{eq:M-m} holds with 
\[
M-m\le \log(B_{2}/B_{1})+C_0 q^{-1}
\]
Finally, \eqref{eq:|F|} follows from \eqref{eq:M-m} by exponentiating. 
\end{proof}

Integrating the previous result over a small set of $x$ yields the following localization estimate for the $L^{2}$ norm of $F$. 

\begin{prop}\label{prop:B1B2}
There exists an absolute constant $C_{1}>0$ with the following property: 
Let $F$ be an analytic function on a neighborhood of $\Rect(q)$ with $0< q\leq q^*$, and $F$ not identically equal to zero.     Denote 
$$ B_{1}:=\|F\|_{L^{2}([1/4,3/4])} ,\qquad   B_{2}:=\|F\|_{L^{2}(\del \Rect(q))} .$$  For any $J\subset [1/4,3/4]$  some Borel set of positive measure, 
\[
B_{1}\le e^{\frac{C_{1}}{q}}     B_{2}^{1-\kappa} \|F\|_{L^{2}(J)}^{\kappa}
\]
with $0<\kappa\le  e^{-\frac{C_{1}}{q}} \big( \log(1/|J|)\big)^{-1}$. 
\end{prop}
\begin{proof}
We apply Corollary~\ref{cor:muehe} with ${  {3}} H \exp\big(\frac{3\pi}{4q}\big)= \frac12|J|$. Thus, 
\EQ{\label{eq:prop61}
  B_{1}^{K+1} (|J|/2)^{\frac12} & \le  e^{\frac{C_{0}K}{q}}\, B_{2}^{K} \, \|F\|_{L^{2}(J)} \\
  K &:= \exp\big(\frac{9\pi}{8q}\big)  \big[\log(4) + \frac{9\pi}{4q}  +  \exp\big(-\frac{3\pi}{8q}\big) \big( \log({  {12}} e/|J|) + \frac{3\pi}{4q}\big) \big]  
} 
or 
\EQ{\label{eq:prop62}
B_{1}\le e^{\frac{C_{0}}{q}}  (|J|/2)^{-\frac{\kappa}{2}}   B_{2}^{1-\kappa} \|F\|_{L^{2}(J)}^{\kappa},\qquad \kappa \le  (1+K)^{-1}.
}
We write $\kappa \le  (1+K)^{-1}$ instead of $\kappa =  (1+K)^{-1}$, since we may increase the value of $K$. 
One checks that 
\EQ{\label{eq:prop63}
\log{( (|J|/2)^{-\frac{\kappa}{2}} )}
\leq &\frac{\log{(2/|J|)}}{\exp\big(\frac{9\pi}{8q}\big)  \big[\log(4) + \frac{9\pi}{4q}  +  \exp\big(-\frac{3\pi}{8q}\big) \big( \log({  {12}} e/|J|) + \frac{3\pi}{4q}\big)\big]}\\
\leq &\exp\big(-\frac{3\pi}{4} \big)<0.1,
} 
uniformly in $0<q<1$ and in $|J|$. 
Note that 
\EQ{\nn
K
\leq &
\begin{cases}
\exp\big(\frac{9\pi}{8q}\big)  \big[\log(4) + \frac{9\pi}{4q}  +  \exp\big(-\frac{3\pi}{8q}\big) \big(\log{(12e)}+ \frac{3\pi}{2q}\big) \big],\quad \text{if } \log{2}\leq \log(1/|J|)<\frac{3\pi}{4q}\\
8\exp\big(\frac{9\pi}{8q}\big) \big[1+\exp\big(-\frac{3\pi}{8q}\big)\big]\log(1/|J|),\qquad\qquad \text{ if } \max\big(\log{2}, \frac{3\pi}{4q}\big)\leq \log(1/|J|)
\end{cases}\\
\leq &e^{\frac{C_2}{q}} \log(1/|J|)-1,
}
for some absolute constant $C_2>0$.
Taking $C_1:=\max{(2C_0, C_2)}$ and 
\[K_0:=e^{\frac{C_1}{q}} \log(1/|J|).\]
We conclude from \eqref{eq:prop61}, \eqref{eq:prop62} and \eqref{eq:prop63} with the estimate $K\leq K_0-1$ that
\[
B_{1}\le e^{\frac{C_{0}}{q}+0.1}   B_{2}^{1-\kappa} \|F\|_{L^{2}(J)}^{\kappa}\leq e^{\frac{C_1}{q}} B_{2}^{1-\kappa} \|F\|_{L^{2}(J)}^{\kappa},\qquad \kappa \le  K_{0}^{-1}
\]
as claimed.  
\end{proof}

We next apply Proposition~\ref{prop:B1B2} to a band limited $L^2$ function in order to obtain the main result of this section. 

\begin{prop}\label{prop:loc}
Fix $\lambda\in (0,\frac12]$ and for each integer $n$  
let $I_n\subset [n,n+1]$ be some Borel set with $|I_n|=\lambda$. 
Let $f\in L^2(\R)$ be band-limited, i.e., $\hat{f}$ is  of compact support. Then  for each $0<q\le q^*$ 
\EQ{
\label{eq:L2loc}
\|f\|_{L^2(\R)}^2 \le 12 \, e^{\frac{10C_{1}}{q}}  \Big( \sum_n \|f\|_{L^2(I_n)}^2\Big)^\kappa   \| e^{2\pi q|\xi|} \hat{f}(\xi) \|_{L^2(\R)}^{2(1-\kappa)} 
}
with $0<\kappa\le  e^{-\frac{5C_{1}}{q}} ( -\log \lambda)^{-1}$, and $C_1, q^*$ are as in  Proposition~\ref{prop:B1B2}. 
\end{prop}
\begin{proof}
Let $F$ be the entire function with $F=f$ on the real line. Fix $0\le t\le 1$ and define $\Rect_{n,t}(q)$ to be the rectangle with vertices $n-1-t\pm iq, n+2+t \pm iq$. 
We claim that by Proposition~\ref{prop:B1B2} we have 
\EQ{\label{eq:Rectnt}
\|f\|_{L^{2}([n,n+1])} \le e^{\frac{5C_{1}}{q}}    \|F\|_{L^{2}(\del \Rect_{n,t}(q))}^{1-\kappa} \|f\|_{L^{2}(I_n)}^{\kappa}
} 
with  $\kappa\le  e^{-\frac{5C_{1}}{q}}  \big( \log((3+2t)/|I_n|)\big)^{-1}$. To see this, we set $n=0$ without loss of generality,    translate $\Rect_{n,t}(q)\to \Rect_{n,t}(q) + 1+t$, and dilate $z\mapsto z/(3+2t)$. 
After these operations, the transformed interval $I_0$ lies in $$[(1+t)/(3+2t), (2+t)/(3+2t)]\subset [1/4, 3/4],$$  and the height $q$ becomes $q/(3+2t)\ge q/5$, whence the claim. 

Squaring, summing, and applying H\"{o}lder's inequality yields
\[
\|f\|_{L^2(\R)}^2 \le e^{\frac{10C_{1}}{q}}  \Big( \sum_n  \|F\|_{L^{2}(\del \Rect_{n,t}(q))}^2\Big) ^{1-\kappa} \Big( \sum_n \|f\|_{L^{2}(I_n)}^2\Big)^{\kappa}
\]
Let $\Erw$ denote the expected value with respect to $0\le t\le 1$, uniformly distributed.   On the one hand,  taking expectations of the previous line yields 
\EQ{\label{eq:Erw}
\|f\|_{L^2(\R)}^2 \le e^{\frac{10C_{1}}{q}}  \Big( \sum_n  \Erw \|F\|_{L^{2}(\del \Rect_{n,t}(q))}^2\Big) ^{1-\kappa} \Big( \sum_n \|f\|_{L^{2}(I_n)}^2\Big)^{\kappa}
}
On the other hand, since 
\EQ{\label{eq:onele5}
\sup_{0\le t\le 1} \sum_n  \one_{[ n-1-t, n+2+t)} \le 5
}
we have 
\EQ{ \label{eq:erw} 
\sum_n  \Erw \|F\|_{L^{2}(\del \Rect_{n,t}(q))}^2 &\le  5 \| F(\cdot +iq)\|_{L^2(\R)}^2 +  5 \| F(\cdot - iq)\|_{L^2(\R)}^2 \\
&\qquad + 2\sum_n \int_0^1 \int_{-q}^q { {| F(n-t+is) |^2}} \, ds dt 
}
Since $ \| F(\cdot \pm iq)\|_{L^2(\R)}  =  \| e^{\pm 2\pi q \xi} \hat{f}(\xi) \|_{L^2(\R)}$, and 
\EQ{ \nn 
& \sum_n \int_0^1 \int_{-q}^q { {| F(n-t+is) |^2}} \, ds dt   = \int_\R \int_{-q}^q { {|F(x+is)|^2}} \, ds dx  \\
&= { {\int_{-q}^q \int_\R  e^{4\pi s\xi}  |\hat{f}(\xi)|^2   \, d\xi ds \le  2q \| e^{2\pi q|\xi|} \hat{f}(\xi)\|_{L^2(\R)}^2 }}
}
Assuming as we may that $q^*\le \frac12$ we infer from \eqref{eq:erw} that 
\[
\sum_n  \Erw \|F\|_{L^{2}(\del \Rect_{n,t}(q))}^2 \le 12 \| e^{2\pi q|\xi|} \hat{f}(\xi)\|_{L^2(\R)}^2
\]
Inserting this into \eqref{eq:Erw} concludes the proof. 
\end{proof}

\section{$L^2$ localization in higher dimensions}\label{sec:locdim}

Our goal is to prove a version of Proposition~\ref{prop:loc} for band-limited functions $f\in  L^2(\R^d)$, $d\ge2$. For the sake of simplicity, we first limit ourselves to $d=2$ and begin with a Cartan-type estimate for functions on $\D\times \D$ which are subharmonic relative to each variable. 

We begin with the definition of a Cartan-$2$ set, cf.~\cite[Definition 8.1]{GS1} and~\cite[Definition 2.12]{GS2}. 

\begin{definition}\label{def:Cart2}
We say that $\calB\subset\C^2$ is a Cartan-$2$ set with parameter $H>0$ if for all $(z_1,z_2)\in \calB$ one has either
\begin{itemize}
\item $z_1\in \bigcup_j D(\zeta_j, s_j) $ with $\sum_j s_j\le 5H$  \text{\ \ or for all other\ \ }$z_1$ one has 
\item $z_2 \in \bigcup_k D(w_t, t_k) $ with $\sum_k t_k\le 5H$ and $(w_k,t_k)$ depend on $z_1$. 
\end{itemize} 
\end{definition} 

Of particular relevance to us with be the fact that a Cartan-$2$ set has a real ``trace" of small measure.  

\begin{lemma}
\label{lem:smalltrace}
Let $\calB\subset \prod_{j=1}^2 D(z_{j,0},1)$ be a Cartan-$2$ set with parameter $H>0$. Then 
\[
| \calB\cap \R^2| \le 40 H
\]
\end{lemma}
\begin{proof}
Follows from Fubini and $|D(\zeta, s)\cap \R|\le 2s$ for all $\zeta\in\C$. 
\end{proof}

We can now formulate a Cartan-type bound for pluri-subharmonic functions.  

\begin{lemma}
\label{lem:Cart2}
Let $v: \ol{\D\times \D} \to [-\infty,\infty)$ be continuous so that $v=v(z_1,z_2)$ is separately subharmonic in each variable. Suppose for $0<\rho<r<1$ 
\EQ{
\label{eq:max2r}
\max_{|z_1|\le r,|z_2|\le r} \int_{\S^1\times \S^1} v(e(\theta_1), e(\theta_2)) \, P_{z_1}(d\theta_1) P_{z_2}(d\theta_2) \le M
}
and 
\EQ{
\label{eq:max2rho}
\max_{|z_1|\le \rho,|z_2|\le \rho }  v(z_1, z_2) \ge m
}
Let $\rho=r(1-3\delta)$ with $0<\delta<\frac13$. Then for any $0<H\le 1$ one has 
\EQ{\label{eq:vCart2} 
v(z_1,z_2) &\ge m  - (M-m)  (L+1)^2 \\
L& :=  2  \delta^{-3}\log(2/\delta) +  \delta^{-2} \log(2e/H)   
}
for all $(z_1,z_2)\in r_1\D\times r_1\D\setminus \calB$ where $\calB$ is  a Cartan-$2$ set with 
parameter~$rH$, and $r_1= r(1-2\delta)$. 
\end{lemma}
\begin{proof}
The function 
\EQ{\label{eq:hz1z2}
h(z_1,z_2):=\int_{\S^1\times \S^1} v(e(\theta_1), e(\theta_2)) \, P_{z_1}(d\theta_1) P_{z_2}(d\theta_2)
}
is separately harmonic in each variable, is continuous up to $\partial(\D\times \D)$, and satisfies $v\le h$ pointwise. 
This latter property follows from the pointwise inequalities 
\EQ{\nn
v(z_1,z_2) & \le \int_{  \S^1} v(z_1, e(\theta_2)) \,  P_{z_2}(d\theta_2)  
}
which holds due to harmonicity of the right-hand side in $z_2$, whence 
\EQ{
\label{eq:vH}
v(z_1,z_2) & \le \int_{\S^1 } v(e(\theta_1), z_2) \, P_{z_1}(d\theta_1) \\
&\le \int_{\S^1\times \S^1} v(e(\theta_1), e(\theta_2)) \, P_{z_1}(d\theta_1) P_{z_2}(d\theta_2) =h(z_1,z_2)
}
as claimed. 
Define 
\EQ{\label{eq:tilv}
\tilde v(z_1) := \max_{|z_2|\le \rho} \; v(z_1, z_2)  
}
Then $\tilde v$ is continuous (by uniform continuity), and subharmonic (as the supremum  of a family of subharmonic functions).  
It satisfies $ \tilde v(z_1)\le M$ for all $|z_1|\le r$ by~\eqref{eq:max2r} and~\eqref{eq:vH}, and $\max_{|z_1|\le\rho} \tilde v(z_1)\ge m$. The latter follows from 
\[
v(z_1, z_2)\le \tilde v(z_1) \qquad\forall\; |z_1|\le r, \; |z_2|\le \rho 
\]
and \eqref{eq:max2rho}. 

We apply Corollary~\ref{cor:vRiesz} to $\tilde v$, which requires rescaling from $\D$ to $r\D$. Thus, with $\rho = r(1-3\delta)$, and $r_1= r(1-2\delta)$, 
 \EQ{\label{eq:good z1}
\tilde v(z_1) &\ge m  - (M-m) L =: m^* 
 }
 for all $z_1\in r_{1}\D \setminus  \bigcup_{j} D(\zeta_{j},s_{j}) $  with $\sum_{j} s_{j} \le 5 rH $.  Fix  such a  {\em good} $z_1$. By definition, there exists $z_2^*$ with $|z_2^*|\le \rho$ and  
 $v(z_1, z_2^*)\ge m^*$. On the other hand, $v(z_1, z_2)\le M$ for all $|z_2|\le r$. 

Once again, by Corollary~\ref{cor:vRiesz} rescaled from $\D$ to $r\D$, it follows that 
 \EQ{\label{eq:good z1z2}
  v(z_1,z_2) &\ge m^*  - (M-m^*) L  \\
&\ge m  - (M-m) L (2+L)
 }
 for all $z_2\in r_{1}\D \setminus  \bigcup_{j} D(w_{j},t_{j}) $  with $\sum_{j} t_{j} \le 5 rH $. These disks depend on $z_1$. 
 \end{proof}

By means of Lemma~\ref{lem:Cart2} we establish a two-dimensional analogue of Proposition~\ref{prop:B1B2}.  

\begin{prop}\label{prop:B1B2 2dim}
There exists an absolute constant $C_{1}>0$ with the following property: 
Let $F$ be an analytic function of two variables on a neighborhood of $\Rect(q)\times \Rect(q)$ with $0< q\leq q^*$, and $F$ not identically equal to zero.     Denote 
$$ B_{1}:=\|F\|_{L^{2}([1/4,3/4]\times[1/4,3/4] )} ,\qquad   B_{2}:=\|F\|_{L^{2}(\del \Rect(q)\times \del \Rect(q))} .$$  For any $J\subset [1/4,3/4]\times [1/4,3/4]$  some Borel set of positive measure, 
\[
B_{1}\le e^{\frac{C}{q}}     B_{2}^{1-\kappa} \|F\|_{L^{2}(J)}^{\kappa}
\]
with $0<\kappa\le  e^{-\frac{C}{q}} \big( \log(1/|J|)\big)^{ -2}$ with some absolute constant $C$. 
\end{prop}
\begin{proof}
Set $u(z_1,z_2):=\log|F(z_1,z_2)|$, which is pluri-subharmonic on a neighborhood of $\Rect(q)\times\Rect(q) $. We pull $u$ back to the polydisk $\D\times\D$, and define 
 $$v(z_1,z_2)=u(\Phi_{q}(z_1), \Phi_{q}(z_2))= \log|\tilde F(z_1,z_2)|,\qquad  \tilde F(z_1,z_2) = F(\Phi_{q}(z_1), \Phi_{q}(z_2)) .$$  
 With $h$ defined as in \eqref{eq:hz1z2}, for all $|z_1|,|z_2|\le r$, 
 \EQ{\label{eq:h2dim}
h(z_1,z_2) &= \int_0^1 \int_0^1  v(e(\theta_1),e(\theta_2) ) \, P_{z_1}(d\theta_1) P_{z_2}(d\theta_2) \\
&=  \int_0^1 \int_0^1  \log| \tilde F(e(\theta_1),e(\theta_2) ) |\, P_{z_1}(d\theta_1) P_{z_2}(d\theta_2) \\
&\le  \log\Big (  \int_0^1 \int_0^1  | \tilde F(e(\theta_1),e(\theta_2) ) |\, P_{z_1}(d\theta_1) P_{z_2}(d\theta_2)   \Big)  \\
&\le  \log\Big ( \int_0^1 \int_0^1  | \tilde F(e(\theta_1),e(\theta_2) ) |\, d\theta_1 d\theta_2 \; \big\| \frac{P_{z_1}(d\theta) }{d\theta}\big\|_{\infty} \big\| \frac{P_{z_2}(d\theta) }{d\theta}\big\|_{\infty} \Big) \\
&\le  \log  ( B_{2} ) + 2\log\big( \| \frac{d\theta}{d\sigma}\|_{L^{2}(\del\Rect(q))}     \big) + 2\sup_{|w|\le r}\log\big\| \frac{P_{w}(d\theta) }{d\theta}\big\|_{\infty} \\
& \le \log  ( B_{2} ) + \log\big(Cq^{-1}) +  2\log(2/(1-r))
}
where $d\sigma$ denotes arc length measure on $\del\Rect(q)$, see~\eqref{eq:Mnew}.  By   Lemma~\ref{lem:Phir}, we can apply Lemma~\ref{lem:Cart2} to $v$ with $\rho=1-\exp(-A/q)$ with some absolute constant $A$, $$m=\log B_1,  \; M=\log  ( B_{2} ) +  3Aq^{-1}, \; \delta=\exp(-2A/q), \; r=\rho(1-3\delta)^{-1},$$ 
 and $0<q\le q^*\ll 1$.   
 Thus, for any $H>0$ there exists a Cartan-$2$ set $\calB$ with parameter $H$ such that 
for 
 \[r_1=1-\exp(-A/q)<r(1-2\delta),\]
 and any $(z_1,z_2)\in r_1\D\times r_1\D\setminus \calB$, we have
 \EQ{\nn 
 v(z_1, z_2)   &\ge m  - (M-m) (L+1)^2,
 }
 where
 \EQ{\nn
 L=2e^{\frac{6A}{q}}\log(2e^{\frac{2A}{q}})+&e^{\frac{4A}{q}}\log{(2e/H)}<e^{\frac{8A}{q}}+e^{\frac{4A}{q}}\log{(2e/H)}-1.
 }
 Returning to the original geometry, and analytic function $F$, we conclude the following via Lemmas~\ref{lem:Phir} and~\ref{lem:smalltrace}:  
 with $K:= (e^{\frac{8A}{q}} + e^{\frac{4A}{q}} \log(2e/H))^2$, 
 \EQ{\nn 
 B_1^{K+1} \le e^{\frac{3AK}{q}}  |F(x_1,x_2)|  B_2^{K}, 
}
 for all $(x_1, x_2)\in [1/4,3/4]\times[1/4,3/4]\setminus \calE$, where $\calE\subset\R^2$ and $|\calE|\le e^{\frac{5A}{q}} H$. 
 
 We now pick $H$ so that $e^{\frac{5A}{q}} H = \frac12|J|$, and integrate over $J$, we obtain
 \EQ{\nn
 B_1^{K+1} (|J|/2)^{\frac12}\leq e^{\frac{3AK}{q}} B_2^K \|F\|_{L^2(J)}
 }
 or
 \EQ{\label{eq:B1B2J}
 B_1\leq e^{\frac{3A}{q}} (|J|/2)^{-\frac{\kappa}{2}}B_2^{1-\kappa}\|F\|_{L^2(J)}^{\kappa},\qquad \kappa\leq (1+K)^{-1}.
 }
We write $\kappa\leq (1+K)^{-1}$ instead of $\kappa=(1+K)^{-1}$ since we could increase $K$.
One easily checks that $(|J|/2)^{-\frac{\kappa}{2}}\lesssim 1$, and 
\[
K\leq e^{\frac{C_1}{q}} (\log(1/|J|))^2-1,
\]
with some absolute constant $C_1$. Taking $C:=\max{(4A, C_1)}$, and 
\[K_0:=e^{\frac{C}{q}} (\log(1/|J|))^2.\]
We conclude from \eqref{eq:B1B2J} with the estimate $K\leq K_0-1$ that
\EQ{\nn
B_1\leq e^{\frac{C}{q}} B_2^{1-\kappa}\|F\|_{L^2(J)}^{\kappa},\qquad \kappa\leq K_0^{-1},
}
as claimed.
\end{proof}

In analogy with the one-dimensional case in Proposition~\ref{prop:loc}, we can deduce the following $L^2$ localization result.

\begin{prop}\label{prop:loc2dim}
Fix $\lambda\in (0,\frac12]$ and for each integers $n_1,n_2$  
let $$I_{n_1,n_2}\subset [n_1,n_1+1]\times [n_2,n_2+1]$$ be some Borel set with $|I_{n_1,n_2}|=\lambda$. 
Let $f\in L^2(\R^2)$ be band-limited, i.e., $\hat{f}$ is  of compact support. Then  for each $0<q\le q^*$ 
\EQ{
\label{eq:L2loc 2dim}
\|f\|_{L^2(\R^2)}^2 \le   e^{\frac{2C}{q}}  \Big( \sum_{(n_1,n_2)\in\Z^2}  \|f\|_{L^2(I_{n_1,n_2})}^2\Big)^\kappa   \| e^{2\pi q(|\xi_1|+|\xi_2|)} \hat{f}(\xi) \|_{L^2(\R^2)}^{2(1-\kappa)} 
}
with $0<\kappa\le  e^{-\frac{C}{q}} ( -\log \lambda)^{ -2}$, and $C$ some absolute constant. 
\end{prop}
\begin{proof}
Let $F$ be the entire function with $F=f$ on $\R^2$. 
Fix $0\leq t_1, t_2\leq 1$ and for $j=1,2$ denote $\Rect_{n,t_j}(q)$ be the rectangle with vertices $n-1-t_j\pm iq$, $n+2+t_j\pm iq$.
We obtain from Proposition \ref{prop:B1B2 2dim} that for any $n_1, n_2\in \Z$:
\EQ{\nn
\|f\|_{L^2([n_1,n_1+1]\times [n_2,n_2+1])}\leq e^{\frac{5C}{q}} \|F\|^{1-\kappa}_{L^2(\partial \Rect_{n_1,t_1}(q) \times \partial \Rect_{n_2,t_2}(q))} 
\|f\|^{\kappa}_{L^2(I_{n_1,n_2})},
}
with $\kappa\leq e^{-\frac{5C}{q}}(\log((3+2t_1)(3+2t_2)/|I_{n_1,n_2}|))^{-2}$, and $C$ being the absolute constant in Proposition \ref{prop:B1B2 2dim}.
Squaring, summing, and applying H\"older's inequality, we have
\EQ{\nn
\|f\|_{L^2(\R^2)}^2 &\leq e^{\frac{10C}{q}} \Big(\sum_{(n_1,n_2)\in \Z^2} \|F\|^{2}_{L^2(\partial \Rect_{n_1,t_1}(q) \times \partial \Rect_{n_2,t_2}(q))} \Big)^{1-\kappa}\\
&\qquad
\Big(\sum_{(n_1,n_2)\in \Z^2} \|f\|_{L^2(I_{n_1,n_2})}^2\Big)^{\kappa}.
}
Taking expectation of the previous line with respect to $0\leq t_1,t_2\leq 1$, we obtain
\EQ{\label{eq:Erw'}
&\|f\|_{L^2(\R^2)}^2\\
\leq &e^{\frac{10C}{q}} \Big(\sum_{(n_1,n_2)\in \Z^2}  \Erw_{t_1} \Erw_{t_2} \|F\|^{2}_{L^2(\partial \Rect_{n_1,t_1}(q) \times \partial \Rect_{n_2,t_2}(q))}\Big)^{1-\kappa}
\Big(\sum_{(n_1,n_2)\in \Z^2} \|f\|_{L^2(I_{n_1,n_2})}^2\Big)^{\kappa}.
}
By decomposing each $\partial\Rect_{n,t}(q)$ into its four sides, we decompose 
\EQ{\label{eq:Erwtotal}
\sum_{(n_1,n_2)\in \Z^2}  \Erw_{t_1} \Erw_{t_2} \|F\|^{2}_{L^2(\partial \Rect_{n_1,t_1}(q) \times \partial \Rect_{n_2,t_2}(q))}
}
into the following three parts:
\subsubsection*{Part 1. Vertical and Horizontal mixed terms}
This part contains eight terms, each can be bounded in the same way.
Taking the left vertical side of $\Rect_{n_1,t_1}(q)$ and upper horizontal side of $\Rect_{n_2,t_2}(q)$ for example, we have
\EQ{\nn
&\sum_{(n_1,n_2)\in \Z^2}  \Erw_{t_2} \int_{\R}  \one_{[n_2-1-t_2,n_2+2+t_2)} \Erw_{t_1}  \int_{-q}^q  |F(n_1-1-t_1+is, x_2+iq)|^2 \, ds dx_2\\
\leq &5\sum_{n_1\in \Z} \Erw_{t_1} \int_{\R}\int_{-q}^q  |F(n_1-1-t_1+is, x_2+iq)|^2 \, ds dx_2\\
=& 5\int_{-q}^q \int_{\R^2} |F(x_1+is, x_2+iq)|^2 \, dx_1 dx_2 ds\\
\leq &5\int_{-q}^q \int_{\R^2} e^{4\pi (s\xi_1+q\xi_2)} |\hat{f}(\xi_1, \xi_2)|^2\, d\xi_1 d\xi_2 ds\\
\leq &10q \| e^{2\pi q(|\xi_1|+|\xi_2|)} \hat{f}(\xi) \|_{L^2(\R^2)}^2,
}
in which we used \eqref{eq:onele5} in the first step.
Hence, part 1 contributes in total at most
\EQ{\label{eq:Erwvh}
80q\| e^{2\pi q(|\xi_1|+|\xi_2|)} \hat{f}(\xi) \|_{L^2(\R^2)}^2.
}

\subsubsection*{Part 2. Vertical+Vertical sides}
This part contains four terms. Taking the left vertical sides of $\Rect_{n_1,t_1}(q)$ and $\Rect_{n_2,t_2}(q)$ for example, we have
\EQ{\nn
&\sum_{(n_1,n_2)\in \Z^2}  \Erw_{t_1} \Erw_{t_2} \int_{-q}^q \int_{-q}^q  |F(n_1-1-t_1+is_1, n_2-1-t_2+is_2)|^2 \, ds_1 ds_2\\
=&\int_{-q}^q \int_{-q}^q \int_{\R^2}  |F(x_1+is_1, x_2+is_2)|^2 \, dx_1 dx_2 ds_1 ds_2\\
\leq &4q^2  \| e^{2\pi q(|\xi_1|+|\xi_2|)} \hat{f}(\xi) \|_{L^2(\R^2)}^2.
}
Hence, part 2 contributes in total at most
\EQ{\label{eq:Erwvv}
16 q^2 \| e^{2\pi q(|\xi_1|+|\xi_2|)} \hat{f}(\xi) \|_{L^2(\R^2)}^2.
}
\subsubsection*{Part 3. Horizontal+Horizontal sides}
This part also contains four terms. Taking the upper horizontal sides of $\Rect_{n_1,t_1}(q)$ and $\Rect_{n_2,t_2}(q)$ for example, we have
\EQ{\nn
\sum_{(n_1,n_2)\in \Z^2} &\!\!\! \Erw_{t_1} \Erw_{t_2} \int_{\R^2} \one_{[n_1-1-t_1, n_1+2+t_1)} \one_{[n_2-1-t_2, n_2+2+t_2)} |F(x_1+iq, x_2+iq)|^2 dx_1 dx_2\\
\leq &25 \int_{\R^2} |F(x_1+iq, x_2+iq)|^2 dx_1 dx_2\\
\leq &25 \| e^{2\pi q(|\xi_1|+|\xi_2|)} \hat{f}(\xi) \|_{L^2(\R^2)}^2,
}
in which we used \eqref{eq:onele5} in the first step.
Hence, the contribution of part (3) is at most
\EQ{\label{eq:Erwhh}
100 \| e^{2\pi q(|\xi_1|+|\xi_2|)} \hat{f}(\xi) \|_{L^2(\R^2)}^2.
}
Plugging the estimates in \eqref{eq:Erwvh}, \eqref{eq:Erwvv} and \eqref{eq:Erwhh} into \eqref{eq:Erwtotal}, we obtain
\EQ{\label{eq:Erwtotal'}
\sum_{(n_1,n_2)\in \Z^2}  \Erw_{t_1} \Erw_{t_2} \|F\|^{2}_{L^2(\partial \Rect_{n_1,t_1}(q) \times \partial \Rect_{n_2,t_2}(q))}
\leq &(4q+10)^2 \| e^{2\pi q(|\xi_1|+|\xi_2|)} \hat{f}(\xi) \|_{L^2(\R^2)}^2\\
\leq 144 \| e^{2\pi q(|\xi_1|+|\xi_2|)} \hat{f}(\xi) \|_{L^2(\R^2)}^2,
}
for $q\leq 1/2$.
Plugging \eqref{eq:Erwtotal'} into \eqref{eq:Erw'} yields
\EQ{\nn
\|f\|_{L^2(\R^2)}^2\leq 144 e^{\frac{10C}{q}}\Big(\sum_{(n_1,n_2)\in \Z^2} \|f\|_{L^2(I_{n_1,n_2})}^2\Big)^{\kappa}  \|e^{2\pi q(|\xi_1|+|\xi_2|)} \hat{f}(\xi) \|_{L^2(\R^2)}^{2(1-\kappa)},
}
as claimed.
\end{proof}

In general dimensions one can proceed similarly. First, we inductively define Cartan sets in higher dimensions. 

\begin{definition}\label{def:Cart2d}
We say that $\calB\subset\C^2$ is a Cartan-$d$ set with parameter $H>0$ if for all $(z_1,z_2,\ldots, z_d)\in \calB$ one has either
\begin{itemize}
\item $z_1\in \bigcup_j D(\zeta_j, s_j) $ with $\sum_j s_j\le 5H$  \text{\ \ or for all other\ \ }$z_1$ one has 
\item $(z_2,\ldots, z_d)$ belongs to a Cartan-$(d-1)$ set with parameter $H>0$ depending on $z_1$. 
\end{itemize} 
\end{definition} 

By arguments analogous to those used above for $d=2$, one can exploit these Cartan sets in higher dimensions 
 to obtain the following result. We leave the details to the reader. Throughout,  we let $C(d)\geq 1$ be a constant depending only on the dimension $d$. It is allowed to change its values from line to line.

\begin{prop}\label{prop:locddim}
Fix $\lambda\in (0,\frac12]$ and for each integer vector $n=(n_1,\ldots,n_d)\in\Z^d$, $d\ge2$,   
let $$I_{n}\subset \prod_{j=1}^d [n_j,n_j+1) $$ be some Borel set with $|I_{n}|=\lambda$. 
Let $f\in L^2(\R^d)$ be band-limited, i.e., $\hat{f}$ is  of compact support. Then  for each $0<q\le q^*=q^*(d)\ll 1$ 
\EQ{
\label{eq:L2loc ddim}
\|f\|_{L^2(\R^d)}^2 \le   e^{\frac{2C(d)}{q}}  \Big( \sum_{n\in\Z^d}  \|f\|_{L^2(I_{n})}^2\Big)^\kappa   \| e^{2\pi q |\xi|_1} \hat{f}(\xi) \|_{L^2(\R^d)}^{2(1-\kappa)} 
}
with $0<\kappa\le  e^{-\frac{C(d)}{q}} ( -\log \lambda)^{ -d}$,   $C(d)\geq 1$ some absolute constant depending on $d$, and $|\xi|_1=\sum_j|\xi_j|$. 
\end{prop}

As a precursor to the results of the next section, which involve $L^2$ functions with Fourier support in thin sets,  we now establish an uncertainty principle for $L^2(\R^d)$ functions under a quantitative decay assumption on their Fourier transforms. 

\begin{cor}
\label{cor:UPtheta}
Let $\Theta(\xi)=\Theta(|\xi|_1)=(\log(2+|\xi|_1))^{-\alpha}$, $0<\alpha<1$.  
Let $\calS:=\bigcup_{n\in\Z^d}I_n$ be as in Proposition~\ref{prop:locddim}. Then 
\EQ{
\label{eq:UPtheta} 
\|f\|_2 \le C(d,\alpha,A,\lambda) \|f\|_{L^2(\calS)}
}
for all $f\in L^2(\R^d)$ with $\| e^{\Theta(\xi)|\xi|_1}\, \hat{f} \|_{L^2(\R^d)} \le A \|f\|_{L^2(\R^d)}$. 
\end{cor}
\begin{proof}
With $0<q$ small to be determined, we fix $R\ge 1$ so   that $2\pi q=\Theta(R)$. 
Split $f=f_1+f_2$, $\hat{f_1}(\xi)= \hat{f}(\xi)\one_{[|\xi|_1\le R]}$. Then by~\eqref{eq:L2loc ddim}, and since $2\pi q\le \Theta(\xi)$ for $|\xi|_1\le R$, 
\EQ{\nn 
\|f_1\|_2^2 &\le e^{\frac{2C(d)}{q}} \|f_1\|_{L^2(\calS)}^{2\kappa} \| e^{\Theta(\xi)|\xi|_1} \hat{f}_1\|_2^{2(1-\kappa)} \\
&\le e^{\frac{2C(d)}{q}} \|f_1\|_{L^2(\calS)}^{2\kappa} (A\|f\|_2)^{2(1-\kappa)}
}
with 
\[
\kappa =   e^{-\frac{C(d)}{q}} ( -\log \lambda)^{ -d} = e^{-\frac{2\pi C(d)}{\Theta(R)}} ( -\log \lambda)^{ -d}
\]
Moreover, since 
\EQ{\nn
\|f\|_2^2 & = \|f_1\|_2^2 + \|f_2\|_2^2 \\
& \le e^{\frac{2C(d)}{q}} (\|f\|_{L^2(\calS)} + \|f_2\|_2)^{2\kappa} (A\|f\|_2)^{2(1-\kappa)} + \|f_2\|_2^2 } 
and 
\[
\|f_2\|_2 \le e^{-\Theta(R)R} \| e^{\Theta(\xi)|\xi|_1} \hat{f}\|_2 \le Ae^{-\Theta(R)R} \|f\|_2  \le \frac12\|f\|_2
\]
where we chose $R$ large enough depending on~$A\ge 1$. It follows that
\[
\|f\|_2^2 \le 2 e^{\frac{2C(d)}{q}} (\|f\|_{L^2(\calS)} +Ae^{-\Theta(R)R} \|f\|_2)^{2\kappa} (A\|f\|_2)^{2(1-\kappa)}
\]
whence
\EQ{\nn
\|f\|_2  & \le 2^{\frac{1}{2\kappa}} A^{\frac{1-\kappa}{\kappa} } e^{\frac{C(d)}{\kappa q}} (\|f\|_{L^2(\calS)} +Ae^{-\Theta(R)R} \|f\|_2)  \\
& = 2^{\frac{1}{2\kappa}} A^{\frac{1-\kappa}{\kappa} } e^{\frac{C(d)}{\kappa q}}  \|f\|_{L^2(\calS)} +  \exp\big( - T(R) \big) \|f\|_2
}
 with
 \EQ{\nn 
T(R) &  = \Theta(R)R- \frac{C(d)}{\kappa q} - \kappa^{-1} \log(\sqrt{2}A) \\
&    = \Theta(R)R -  \Big( \frac{2\pi C(d)}{\Theta(R)} + \log(\sqrt{2}A) \Big)\, e^{\frac{2\pi C(d)}{\Theta(R)}} ( -\log \lambda)^{ d}  
}
In addition to $2A\le e^{\Theta(R)R} $ we require that $T(R)\ge1$. These conditions hold for sufficiently large~$R$. 
\end{proof}

The proof of the corollary gives an explicit and effective dependence of the constant $C(d,\alpha,A,\lambda)$ on $A, \lambda$, but we have no need for it. 
Corollary~\ref{cor:UPtheta} follows (perhaps with a different dependence on the constants) from a quantitative version of the Logvinenko-Sereda theorem. 
The results in the next section, however, do not. 

\section{Uncertainty principle with thin Fourier support}

We begin with the concept of a damping function. 

\begin{definition}\label{def:damp}
Let $\Theta$ be as in Corollary~\ref{cor:UPtheta}, with $\alpha\in (0,1)$ fixed.  
Let $Y\subset \R^d$. We say that $Y$ {\em admits a damping function} with parameters $c_1, c_2, c_3$, all falling into the interval~$(0,1)$, if there exists a function $\psi \in L^2(\R^d)$ satisfying 
\begin{itemize}
\item $\supp(\psi)\subset [-c_1,c_1]^d$ 
\item $\|\what{\psi}\|_{L^2([-1, 1]^d)}\ge c_2$ 
\item $| \what{\psi}(\xi)|  \le \langle \xi\rangle^{-d}\ $ for all $\xi\in\R^d$
\item $| \what{\psi}(\xi)|  \le \exp\big( -c_3 \Theta(|\xi|_1) |\xi|_1)\ $ for all $\xi\in Y$
\end{itemize}
Here $\langle \xi\rangle = (1+|\xi|_{2}^{2} )^{\frac12}$, $|\xi|_{2}^{2}:=\sum_{j=1}^d \xi_j^2$.
\end{definition}

\begin{lemma}
\label{lem:mainlem}
Fix $c_{1}\in (0,\frac12]$ and for each integer vector $n=(n_1,\ldots,n_d)\in\Z^d$, $d\ge2$,   
let $$I_{n}\subset \prod_{j=1}^d [n_j,n_j+1)$$ be a square with side length~$2c_{1}$. 
Define $\calS:=\bigcup_{n\in\Z^d}I_n$. 
Suppose $Y\subset \R^{d}$  is such that $Y+{{[-2,2]^d}}$ admits a damping function with parameters $c_1$, and $c_2, c_3  \in (0,1)$. 
Then every $f\in L^{2}(\R^{d})$ with 
$\supp(\hat{f})\subset Y$ satisfies
\EQ{
\label{eq:hatfkappa}
\|\hat{f}\|_{L^{2}([-1,1]^d)}^{2} &\le   C(d)c_{2}^{-2}\, { {\la R\ra^{2d}\, e^{\frac{4\pi C(d)}{c_{3}\Theta(R)}}}}\,   \big(  \| \one_{\calS}f \|_{H^{-d}}^{2\kappa}  \|f \|_{H^{-d}}^{2(1-\kappa)}  \\
& \qquad + \exp( -2c_3\kappa \Theta(R)R)\; \| f \|_{H^{-d}}^{2} \big)
}
and $\kappa=e^{-\frac{2\pi C(d)}{c_{3}\Theta(R)}} ( -d\log c_{1} )^{ -d}$, provided 
$R\ge (2d/c_3)^2$ and $0<c_3\le { {c_3^*(d):=}}2\pi q_*$ where $q_*$ is as in Proposition~\ref{prop:locddim}. 
\end{lemma}
\begin{proof}  Let $\eta\in { {[-2,2]}}^{d}$.  
Set $f_{\eta}(x):=e^{2\pi ix\cdot\eta}f(x)$, and $g_{\eta}:=f_{\eta}\ast\psi$ where $\psi$ is the damping function as in Definition~\ref{def:damp} associated with $Y+{{[-2,2]^d}}$. Split 
\EQ{
g_{\eta} &= g_{1}+ g_{2}, \\
\supp(\what{g_{1}}) &\subset \{ \xi\in\R^{d} \::\: |\xi|_{1}\le R\} \\
\supp(\what{g_{2}}) &\subset \{ \xi\in\R^{d} \::\: |\xi|_{1}> R\}
}
where $2\pi q =  c_{3}\Theta(R)$. 
Note that our assumption $c_3\leq 2\pi q_*$ guarantees that $q\leq q_*$ holds for any $R\geq 1$.
Note also that since $\supp(\psi)\subset [-c_1,c_1]^d$, we have 
\linebreak $\one_{\calS'}g_{\eta} = \one_{\calS'} (\one_{\calS}f_{\eta}\ast\psi)$ where $\calS':=\bigcup_{n\in\Z^d}I_n'$ with $I_{n}'$ a square with the same center as $I_{n}$, but half the side length. 
By Proposition~\ref{prop:locddim} with $\lambda=c_{1}^{d}$ one has 
\EQ{ \label{eq:geta}
\| g_{\eta}\|_{2}^{2} &= \|g_{1}\|_{2}^{2} + \|g_{2}\|_{2}^{2} \\
&\le \exp(2C/q) \big(  \| g_{\eta}\|_{L^{2}(\calS')} + \|g_{2}\|_{2}\big)^{2\kappa} \| e^{2\pi q |\xi|_{1}}\what{g_{1}}\|_{2}^{2(1-\kappa)}  + \|g_{2}\|_{2}^{2}
}
with $$0<\kappa\le  e^{-\frac{C(d)}{q}} ( -d\log c_{1} )^{ -d} = e^{-\frac{2\pi C(d)}{c_{3}\Theta(R)}} ( -d\log c_{1} )^{ -d}, $$   
$C(d)$ some absolute constant. 
By construction, {{$\supp(\what{f}_{\eta})\subset Y+{\eta}\subset Y+[-2,2]^d$, hence}}
\[
|\what{g_{\eta}}(\xi)|\le |\what{f_{\eta}}(\xi)| \exp\big( -c_3 \Theta(|\xi|_1) |\xi|_1) \qquad \forall\; \xi\in\R^{d}
\]
whence 
\EQ{\nn 
\| e^{2\pi q |\xi|_{1}}\what{g_{1}}\|_{2} &= \|e^{c_3 \Theta(R) |\xi|_1} \what{g_1}\|_2 \le \sup_{|\xi|_{1}\le R} \la \xi\ra^{d} \; \|f_{\eta}\|_{H^{-d}} \le \la R\ra^{d}\; \|f_{\eta}\|_{H^{-d}}  \\
\|g_{2}\|_{2} &\le \sup_{|\xi|_{1}\ge R}\exp( -c_3 \Theta(|\xi|_{1}) |\xi|_{1})\la \xi\ra^{d} \; \| f_{\eta}\|_{H^{-d}} \\ 
&\le \exp( -c_3 \Theta(R)R)\la R\ra^{d} \; \| f_{\eta}\|_{H^{-d}}
}
where we used that $|\xi|_{2}\le |\xi|_{1}$, and that $r\mapsto \exp( -c_3 \Theta(r) r)\la r\ra^{d} $ is decreasing for large $r$. To be specific,
\EQ{\nn 
\exp( -c_3 \Theta(r) r)\la r\ra^{d} &= \exp(-h(r))\\
h(r) &= c_3\, (\log(2+r))^{-\alpha}\, r - \frac{d}{2} \log(1+r^2)
}
Differentiating, we obtain
\EQ{\nn
h'(r) &= c_3\, (\log(2+r))^{-\alpha} \big[  1 - \frac{\alpha r}{2+r} (\log(2+r))^{-1}\big] - \frac{dr}{1+r^2}\\
&\ge \frac{c_3}{2}\, (\log(2+r))^{-\alpha} - dr^{-1}\ge \frac{c_3}{2}\, (\log(2+r))^{-1} - dr^{-1}
}
where we used that $\frac{\alpha r}{2+r} (\log(2+r))^{-1} \le \frac12$ for all $r\ge0$.  One has $u> \log(2+u^2)$ for $u\ge 2$, say. Hence, if $r\ge  (2d/c_3)^2$, then 
\[
\frac{c_3}{2}\, (\log(2+r))^{-1} - dr^{-1}>0
\] 
and thus $h'(r)>0$. So it suffices to assume that $R\ge (2d/c_3)^2$.  

Inserting these bounds into~\eqref{eq:geta} yields
\EQ{ \nn
\| g_{\eta}\|_{2}^{2} 
&\le e^{\frac{2C(d)}{q}} \big(  \| \one_{\calS}f_{\eta}\|_{H^{-d}} +   \exp( -c_3 \Theta(R)R)\la R\ra^{d} \; \| f_{\eta}\|_{H^{-d}} \big)^{2\kappa} 
\big(\la R\ra^{d}\; \|f_{\eta}\|_{H^{-d}}\big)^{2(1-\kappa)}\\
&  + \exp( -2c_3 \Theta(R)R)\la R\ra^{2d} \; \| f_{\eta}\|_{H^{-d}}^{2}
}
Since  $\sup_{\eta\in { {[-2,2]}}^{d}}\| f_{\eta}\|_{H^{-d}} \le C(d) \|f\|_{H^{-d}}$, we can simplify this further: 
\EQ{
\| g_{\eta}\|_{2}^{2} 
&\le C(d)\la R\ra^{2d}\, e^{\frac{4\pi C(d)}{c_{3}\Theta(R)}} \big(  \| \one_{\calS}f \|_{H^{-d}}^{2\kappa}  \|f \|_{H^{-d}}^{2(1-\kappa)}  +   \exp( -2c_3\kappa \Theta(R)R)\; \| f \|_{H^{-d}}^{2} \big).
}
Finally, 
\EQ{\nn
\|\what{f}\|_{L^{2}([-1,1]^{d})}^{2} &\le c_{2}^{-2}\int_{[-1,1]^{d}}|\what{f}(\zeta)|^{2}\, d\zeta \, \int_{[-1,1]^{d}} |\what{\psi}(\xi)|^{2}\, d\xi \\
&\le c_{2}^{-2} \int_{[-1,1]^{d}} \int_{[-2,2]^{d}} |\what{f}(\xi-\eta)|^{2} |\what{\psi}(\xi)|^{2}\,  d\eta d\xi \\
&\le c_{2}^{-2}  \int_{[-2,2]^{d}} \int_{\R^{d}} |\what{f}(\xi-\eta)|^{2} |\what{\psi}(\xi)|^{2}\,   d\xi d\eta\\
&= c_{2}^{-2}  \int_{[-2,2]^{d}} \|g_{\eta}\|_{2}^{2}\, d\eta
}
and we are done. 
\end{proof}

We now remove the localization in Fourier space on the left-hand side of~\eqref{eq:hatfkappa} in order to obtain the main result of this section. 

\begin{cor}
\label{cor:UPY} 
Fix $c_{1}\in (0,\frac12]$ and for each integer vector $n=(n_1,\ldots,n_d)\in\Z^d$, $d\ge2$,   
let $$I_{n}\subset \prod_{j=1}^d [n_j,n_j+1)$$ be a square with side length~$2c_{1}$. 
Define $\calS:=\bigcup_{n\in\Z^d}I_n$. 
Suppose $Y\subset [-\alpha_1,\alpha_1]^d \subset \R^{d}$ with $\alpha_1\ge1$  is such that $Y+{{[-2,2]^d}}+\eta$ admits a damping function with parameters $c_1$, and $c_2, c_3  \in (0,1)$ for each $\eta\in {{[-\alpha_1-1,\alpha_1+1]}}^{d}$. Assume further that $0<c_3<c_3^*(d)\ll1$, with $c_3^*(d)$ be as in Lemma \ref{lem:mainlem}.
Then every $f\in L^{2}(\R^{d})$ with 
$\supp(\hat{f})\subset Y$ satisfies
\EQ{\label{eq:fYUP}
\|f\|_2\le C_{*} \|f\|_{L^2(\calS)}
}
with constant $C_{*}$ depending only on $d,c_1,c_2,c_3,\alpha$ explicitly as in \eqref{def:Cdc2c3alp}.
\end{cor}
\begin{proof}
Let $\ell\in { {(2\Z)^d}}$ be such that $\ell+{ {[-1,1]^d}}\cap [-\alpha_1,\alpha_1]^d\ne \emptyset$ and define $f_\ell(x):= e^{2\pi ix\cdot\ell} f(x)$ so that $\what{f_\ell}(\xi)=\what{f}(\xi-\ell)$ and  $\supp(\what{f_\ell})\subset Y+\ell$. In order to apply Lemma~\ref{lem:mainlem}, we also need to ensure that $Y+{{[-2,2]^d}}+\ell$ admits a damping function. This, however, follows from our assumptions. 
Hence, for each such~$\ell$, 
 \EQ{
\label{eq:hatfellkappa}
\|\hat{f}\|_{L^{2}([-1,1]^d+\ell)}^{2} &\le   C(d)c_{2}^{-2}\, { {\la R\ra^{2d}\, e^{\frac{4\pi C(d)}{c_{3}\Theta(R)}}}}\,  \big(  \| \one_{\calS}f_\ell \|_{H^{-d}}^{2\kappa}  \|f_\ell \|_{H^{-d}}^{2(1-\kappa)}  \\
& \qquad +   \exp( -2c_3\kappa \Theta(R)R)\; \| f_\ell \|_{H^{-d}}^{2} \big)
}
and $\kappa=e^{-\frac{2\pi C(d)}{c_{3}\Theta(R)}} ( -d\log c_{1} )^{ -d}$, provided $R\ge (2d/c_3)^2$.
Summing over ${ {\ell\in(2\Z)^d}}$, and using H\"older's inequality yields
\EQ{\label{eq:R0}
\|f\|_2^2 
&\le  C(d)c_{2}^{-2}\, { {\la R\ra^{2d}\, e^{\frac{4\pi C(d)}{c_{3}\Theta(R)}}}}\,   
\big(  \| \one_{\calS}f \|_{2}^{2\kappa}  \|f \|_{2}^{2(1-\kappa)} +   \exp( -2c_3\kappa \Theta(R)R)\; \| f  \|_{2}^{2} \big)\\
&=C(d)c_{2}^{-2}\, { {\la R\ra^{2d}\, e^{\frac{4\pi C(d)}{c_{3}\Theta(R)}}}}\| \one_{\calS}f \|_{2}^{2\kappa}  \|f \|_{2}^{2(1-\kappa)}\\
&\qquad+C(d)c_{2}^{-2}\, { {\la R\ra^{2d}\, e^{\frac{4\pi C(d)}{c_{3}\Theta(R)}}}} e^{-2c_3\kappa \Theta(R)R}\| f  \|_{2}^{2}.
}
Suppose further that $R$ satisfies,
\begin{align}\label{eq:R}
R\geq 
R_0(d,c_1,c_2,c_3,\alpha):=
\max
\begin{cases}
&(i).\ \exp{\Big[\Big(\frac{16\pi C(d)}{c_3}\Big)^{\frac{1}{1-\alpha}}\Big]}, \\
&(ii).\ \exp{\Big(4^{\frac{1}{1-\alpha}}\Big)}, \\
&(iii).\ \Big(\frac{(-d\log{c_1})^{d}}{c_3}\Big)^8,\\
&(iv).\ \Big(4\log{\frac{2C(d)}{c_2^2}}\Big)^2, \\
&(v).\ (8d)^4.
\end{cases}
\end{align}
Note that (i), (ii), (iii) of \eqref{eq:R} imply
\EQ{\label{eq:R1}
e^{-\frac{2\pi C(d)}{c_3 \Theta(R)}}(R+2)^{\frac{1}{4}}&\geq 1,\\
\Theta(R) (R+2)^{\frac{1}{8}}&\geq 1,\, \ \text{and}\\
\frac{c_3}{(-d\log{c_1})^d} (R+2)^{\frac{1}{8}}&\geq 1,
}
respectively. Hence multiplying the three inequalities of \eqref{eq:R1} yields
\EQ{\label{eq:R1.5}
c_3\kappa \Theta(R) (R+2)\geq \sqrt{R+2}&, \ \ \text{or}\\
\kappa\geq (c_3\Theta(R)\sqrt{R+2})^{-1}&,
}
and thus
\EQ{\label{eq:R2}
e^{2c_3\kappa \theta(R) R}\ge e^{c_3\kappa \theta(R)(R+2)}\geq e^{\sqrt{R+2}}.
}
One also derives from (iv), (v) and (i) that
\EQ{\label{eq:R3}
\frac{1}{4}\sqrt{R+2}&\geq \log{\frac{2C(d)}{c_2^2}},\\
\frac{1}{2}\sqrt{R+2}&\geq 2d\log(R+2)\geq \log{\la R\ra^{2d}},\, \ \text{and},\\
\frac{1}{4}\sqrt{R+2}&\geq \frac{4\pi C(d)}{c_3 \Theta(R)},
}
respectively. Hence by summing up the three inequalities of \eqref{eq:R3}, and exponentiating, we obtain
\EQ{\label{eq:R4}
e^{\sqrt{R+2}}\geq 2C(d)c_{2}^{-2}\, { {\la R\ra^{2d}\, e^{\frac{4\pi C(d)}{c_{3}\Theta(R)}}}}.
}
Combining \eqref{eq:R2} with \eqref{eq:R4}, we arrive at
\EQ{\nn
C(d)c_{2}^{-2}\, {{\la R\ra^{2d}\, e^{\frac{4\pi C(d)}{c_{3}\Theta(R)}}}}\, e^{-2c_3\kappa \Theta(R)R}\le \frac 12.
}
Thus \eqref{eq:R0} yields
\EQ{\nn 
\|f\|_2  &\le \Big(2C(d)c_{2}^{-2}\, {{\la R\ra^{2d}\, e^{\frac{4\pi C(d)}{c_{3}\Theta(R)}}}} \Big)^{\frac{1}{2\kappa}}  \| \one_{\calS}f \|_{2}.
}
Combining the estimate of $\kappa$ in \eqref{eq:R1.5} with \eqref{eq:R4}, we obtain
\EQ{\nn
\Big(2C(d)c_{2}^{-2}\, {{\la R\ra^{2d}\, e^{\frac{4\pi C(d)}{c_{3}\Theta(R)}}}} \Big)^{\frac{1}{2\kappa}}\leq e^{\frac{c_3 \Theta(R) (R+2)}{2}}.
}

Now we take $R_0$ as in \eqref{eq:R1} and define $R_1$ as follows
\EQ{\label{def:R}
R_1(d,c_1, c_2, c_3, \alpha):=\max{\big((2d/c_3)^2, R_0(d,c_1,c_2,c_3,\alpha)\big)}.
}
Then 
\EQ{\nn
\|f\|_2  &\le C_*(d,c_1,c_2,c_3,\alpha) \| \one_{\calS}f \|_{2},
}
with 
\EQ{\label{def:Cdc2c3alp}
C_*(d,c_1,c_2,c_3,\alpha)=e^{\frac{c_3 \Theta(R_1) (R_1+2)}{2}},
}
as claimed.
\end{proof}

\section{FUP assuming damping functions on $Y$}

In section we prove, by the same iteration as in \cite{BD}, the fractal uncertainty principle for sets $X\subset [-1,1]^{d}$ and $Y\subset [-N,N]^{d}$. On $Y$ we do not impose a geometric condition. Rather, in this section we still restrict ourselves to assuming the existence of damping functions living on~$Y$, as well as on sets derived from $Y$ through translations  and dilations, see Definition~\ref{def:damp}.  On $X$ we impose a certain tree structure ``with gaps'', cf.~\cite[Lemma~2.10]{BD}. 

\begin{definition}
\label{def:porous} 
We say that $X\subset [-1,1]^{d}\subset \R^{d}$ {\em is porous at scale} $L\ge3$ {{{\em with depth $n$}}}, where $L$ is an integer, if the following holds:  denote by $\cC_{n}$ the cubes obtained from $[-1,1]^{d}$ by partitioning it into congruent cubes of side length $L^{-n}$. Thus, $\#\, \cC_{n}=2^{d}L^{nd}$. The condition on $X$ is that for all $Q\in \cC_{n}$ with $Q\cap X\ne\emptyset$, there exists $Q'\in \cC_{n+1}$  so that $Q'\subset Q$ and $Q'\cap X=\emptyset$. 
\end{definition}

{{It is shown in~\cite{BD} that sets $X\subset\R$ obeying the $\delta$-regularity condition on scales $N^{-1}$ to $1$ (see Definition \ref{def:deltareg}) satisfy this porosity property at depth $n$ for all $n\geq 0$ with $L^{n+1}\leq N$.
We include a $d$-dimensional analogy in Appendix \ref{sec:appregular}, see Lemma \ref{lem:BDlemma210}.}}
We can now formulate the Fractal Uncertainty Principle, conditionally on the existence of damping functions in $Y$. As in~\cite{BD} the argument is based on an induction on scales, where at each step a small gain is achieved by means of Corollary~\ref{cor:UPY}. 
Recall that $\alpha\in (0,1)$ is the parameter from the damping function.  

\begin{theorem}
\label{thm:FUP2} 
Let $X\subset [-1,1]^{d}\subset \R^{d}$  be porous at scale $L\ge3$ with depth $n$, {{for all $n\geq 0$ with $L^{n+1}\leq N$}}. 
Suppose $Y\subset [-N,N]^{d}$ is such that for all {{ $n\geq 0$ with $L^{n+1}\leq N$}} one has that for all $$\eta\in [-NL^{-n}-{{3}}, NL^{-n}+{{3}}]^{d}$$ the set 
\EQ{
\label{eq:Y changed}
L^{-n}Y+{{[-4,4]^d}}+\eta
}
admits a damping function with parameters $c_1=(2L)^{-1}\in (0,\frac12]$, and $c_2, c_3  \in (0,1)$. Assume $0<c_{3}< c_{3}^{*}(d)$ as in Corollary~\ref{cor:UPY}. Then there exists $\beta=\beta(L,c_{2},c_{3},d,\alpha)>0$ {{and $\tilde{C}=\tilde{C}(L,c_{2},c_{3},d,\alpha)>0$}} so that any $f\in L^{2}(\R^{d})$ with $\supp(\what{f}\,)\subset Y$ satisfies 
\EQ{\label{eq:FUP}
\| f\|_{L^{2}(X)}\le \tilde{C} N^{-\beta} \|f\|_{L^{2}(\R^{d})}
}
for all $N\ge N_{0}( L,c_{2},c_{3},d,\alpha)$. 
\end{theorem}

\begin{proof}
We pick a nonnegative Schwarz function $\phi$ in $\R^{d}$ with $\supp(\what{\phi})\subset [-1,1]^{d}$ and $\what{\phi}(0)=1$. 
With $T\in \N$ to be determined, we set $\psi(x):= L^{Td}\phi(L^T x)$ so that $\supp(\what{\psi})\subset [-L^T,L^T]^{d}$. 
Let 
\EQ{
\calS_{n} &:=\bigcup_{\substack{Q\in \cC_{n} \\ Q\cap X\ne \emptyset}} Q \\
\calS_{n}^{*} &:= \calS_{n}+[-L^{-n}/10, L^{-n}/10]^d \\
}
and define $\Psi_{n}:= \psi_{n}\ast \one_{\calS_{n+1}^{*}}$ where $\psi_{k}(x):= L^{kd}\psi(L^{k}x)$. 
There exists a constant $C_{\phi}$ depending only on $\phi$ such that  for any $n\geq 0$,
\[
\Psi_n\geq \Big( 1-\frac{C_{\phi}}{L^{T-1}} \Big) \one_X.
\] 
Thus, for all {{$m\ge 1$}}, 
\EQ{\label{eq:PsiX}
\prod_{n=0}^{m-1} \Psi_{n}\ge \Big(1-\frac{C_{\phi}}{L^{T-1}}\Big)^{m}\one_{X}.
}
Moreover, if $Q\in\cC_{n+1}$ with $n\ge 0$ satisfies $Q\cap X=\emptyset$, denote by $Q^{*}$ the cube with the same center as $Q$, but half the side length, i.e., of side length~$L^{-(n+1)}/2$.  
Denote the collection of all such cubes $Q^*$ by $U_{n+1}$.
By the definitions of $S_{n+1}^*$ and $Q^*$, we clearly have
\[S_{n+1}^*\cap \Big(U_{n+1}+[-L^{-(n+1)}/10, L^{-(n+1)}/10]^d \Big)=\emptyset.\]
Then for $x\in U_{n+1}$, and a constant $c_{\phi}$ that depends on $\phi$ only, we have 
\EQ{\label{eq:hole small}
\Psi_n(x) & =\int_{\R^d}\phi_{n+T}(x)\one_{S_{n+1}^*}(x-y)\, dy \\
& =\int_{\R^d}\phi(y) \one_{S_{n+1}^*}(x-L^{-(n+T)}y)\, dy\\
& \leq \int_{\R^d\setminus [-L^{T-1}/10,\ L^{T-1}/10]^d}\phi(y)\, dy\leq \frac{c_{\phi}}{L^{T-1}},
}
uniformly in $n$.

Let $f\in L^{2}(\R^{d})$ with $\supp(\what{f}\,)\subset Y$. Then for $m\geq 1$,
\[
f_{m}:= \prod_{n=0}^{m-1}\Psi_{nT}\cdot f
\]
satisfies 
\EQ{
\label{eq:fmhat supp}
\supp(\what{f_{m}}) &\subset Y + \sum_{n=0}^{m-1} \supp(\what{\psi_{nT}}) \\
&\subset Y + \sum_{n=0}^{m-1} [ - L^{(n+1)T}, L^{(n+1)T}]^{d} = Y + \ell_{m}[-1,1]^{d}
}
where
\[
\ell_{m} := L^T\frac{L^{mT}-1}{L^T-1}.
\]
One has $f_{m+1}=\Psi_{mT}f_{m}$ for all $m\ge0$ with $f_{0}=f$.  
We claim that there exists $\gamma_{0}=\gamma_{0}(L,d,c_{1},c_{2},c_{3}) \in (0,1)$ with 
\EQ{\label{eq:gamma}
\| f_{m+1}\|_{L^2([-1,1]^{d})} \le (1-\gamma_{0}) \|f_{m}\|_{L^2([-1,1]^{d})}
}
Define $g_{m}(x):=  f_{m}(L^{mT}\,x)$. Then 
\EQ{\label{eq:m0}
\supp(\what{g_{m}}) 
&\subset L^{-mT}Y + \ell_{m}L^{-mT} [-1,1]^{d}\\
&\subset L^{-mT}Y+[-2,2]^d,
}
where we used
\[
\ell_{m}L^{-mT} \le   \frac{L^{T}}{L^T-1} \leq 2.
\]
In particular, assuming also that $L^{mT}\le N$, 
\EQ{\nn
\supp(\what{g_{m}}) 
\subset [-NL^{-mT},NL^{-mT}]^{d}+ [-2,2]^{d}= [-NL^{-mT}-2, NL^{-mT}+2]^d,
}
where $NL^{-mT}+2$ will be our parameter $\alpha_1$ in Corollary~\ref{cor:UPY}.

Under this rescaling, the cubes in $\cC_{mT}$ turn into unit cubes.
Assuming further $L^{mT+1}\leq N$,
the porosity condition at scale $L$ with depth $mT$
ensures that we always have a ``missing cube''  of side length $ L^{-1}$ inside. In view of our definition of $Q^*$, we only use the concentric cube of half that side length.  In view of the conditions on $Y$ in the theorem we can apply Corollary~\ref{cor:UPY} to $g_{m}$ to obtain the following: with all norms being taken locally on $[-1,1]^{d}$, and with $U_{mT+1}$, the missing cubes of the next generation as above, 
\EQ{\label{eq:gamma-1}
\| \Psi_{mT}f_{m} \|_{2}^{2} 
&\le \|\Psi_{mT}\|_{\infty}^{2} \| f_{m}\|_{L^{2}([-1,1]^{d}\setminus U_{mT+1})}^{2}
+\|\Psi_{mT}\|_{L^\infty(U_{mT+1})}^{2} \| f_{m}\|_{L^{2}(U_{mT+1})}^{2}\\
&=\|\Psi_{mT}\|_{\infty}^{2} \| f_{m}\|_{L^{2}([-1,1]^{d})}^2
-(1-\|\Psi_{mT}\|_{L^\infty(U_{mT+1})}^{2})\| f_{m}\|_{L^{2}(U_{mT+1})}^{2}\\
&\le \Big(1-C_{*}^{-2}\big(1-{c_{\phi}^2}/{L^{2(T-1)}}\big) \Big)\|f_{m}\|_{2}^{2}
}
To obtain this estimate, we used that 
\[\|\Psi_{mT}\|_{\infty}\le1,\quad \|\Psi_{mT}\|_{L^\infty(U_{mT+1})}\le \frac{c_{\phi}}{L^{T-1}},\] 
and 
\EQ{\nn
\| f_{m}\|_{L^{2}(U_{mT+1})}   & \geq C_*^{-1} \| f_{m}\|_{2}^{2},
}
with $C_{*}=C_{*}(d,L,c_{2},c_{3},\alpha)$ by Corollary~\ref{cor:UPY}.
Choosing
\EQ{\label{def:gamma0}
\gamma_0(T):=\frac{1-{c_{\phi}^2}/{L^{2(T-1)}}}{2C_*^2},
}
and using $(1-x)^{1/2}\leq 1-x/2$ for $0\leq x\leq 1$, we have
\[\Big(1-C_{*}^{-2}\big(1-{c_{\phi}^2}/{L^{2(T-1)}}\big)\Big)^\frac12\leq 1-\gamma_0(T).\]
This establishes the claim~\eqref{eq:gamma}.
 
Applying \eqref{eq:gamma} iteratively and using \eqref{eq:PsiX}, we obtain
\EQ{\label{eq:Xf'}
\|f\|_{L^{2}(X)} & \le \Big(1-\frac{C_{\phi}}{L^{T-1}}\Big)^{-(m+1)} \|\prod_{n=0}^{m} \Psi_n f\|_{L^2(X)} \\
&\le \Big[\Big(1-\frac{C_{\phi}}{L^{T-1}}\Big)^{-1} (1-\gamma_{0}(T)) \Big]^{m+1}  \|f\|_{2} \\
& \le \Big(1-\frac{\gamma_{0}(T)}{2}\Big)^{m+1} \|f\|_{2}.
}
In the last inequality we used
\EQ{\nn
1-\gamma_0(T)\leq 1-\frac{\gamma_0(T)}{2}-\frac{C_{\phi}}{L^{T-1}}\leq \Big(1-\frac{\gamma_0(T)}{2}\Big) \Big(1-\frac{C_{\phi}}{L^{T-1}}\Big),
} 
which requires
\EQ{\label{def:T} 
&L^{T-1}-\frac{c_{\phi}^2}{L^{T-1}}\geq 4C_{\phi} C_*^2,\ \ \text{or}\\
&T\geq T_0(d,L,c_2,c_3,\alpha):=\Bigg\lceil \frac{\log(2C_{\phi}C_*^2+\sqrt{4C_{\phi}^2 C_*^4 +c_{\phi}^2})}{\log L} \Bigg\rceil.
}
Finally, for any $T\geq T_0$, taking $m\in \N$ be such that $L^{mT+1}\leq N<L^{(m+1)T+1}$, \eqref{eq:Xf'} yields \eqref{eq:FUP} with
\EQ{\label{eq:Xf}
\beta=-\frac{\log{(1-\gamma_0(T)/2)}}{T\log{L}},
}
and
\EQ{\label{eq:XfC}
\tilde{C}=\Big(1-\frac{\gamma_0(T)}{2}\Big)^{-1/T}.
}
as claimed.
In the current theorem, we could simply choose $T=T_0$.
The flexibility of choosing $T$ will simplify our computations in our proof of Theorem \ref{thm:main1}.
\end{proof}

\section{Geometry of $Y$ and damping functions} \label{sec:Y}

\subsection{Regular sets}
We will call a set $I=[a_1,b_1]\times [a_2,b_2]\times \cdots \times [a_d,b_d]$ of equal side length a {\em $d$-dimensional cube} in $\R^d$, we denote its side length by $r_I$.

Recall the notion of $\delta$-regularity from \cite[Definition~1.1]{BD}, below is a $d$-dimensional analogy.

\begin{definition}\label{def:deltareg} 
Suppose $X\subset\R^d, X\ne\emptyset$ is closed, and $0<\delta<d$, $C_{R}\ge1$, $0\le \alpha_{0}\le \alpha_{1}\le \infty$. 
Then $X$ is $\delta$-regular on scales $\alpha_{0}$ to $\alpha_{1}$, with constant $C_{R}$, if there exists a Borel measure $\mu_{X}$ with the following properties:
\begin{itemize}
\item $\mu_{X}$ is supported on $X$
\item $\mu_{X}(I)\le C_{R} r_I^{\delta}$ for each $d$-dimensional cube $I$ of side length $\alpha_{0}\le r_I \le \alpha_{1}$
\item $\mu_{X}(I)\ge C_{R}^{-1} r_I^{\delta}$ for each $d$-dimensional cube $I\subset\R^d$, centered at a point in $X$ and of side length 
$\alpha_{0}\le r_I \le \alpha_{1}$
\end{itemize}
\end{definition}
See~\cite[Section 2.2]{BD} for the geometry of such sets in $\R$. 
Loosely speaking, they behave like $\delta$-dimensional fractal sets. 
The properties of $\delta$-regular sets carry over to higher dimensions. 
We include some properties in Appendix \ref{sec:appregular}.

\subsection{Geometry of $Y$ and damping functions}\label{sec:constructionofdamping}

Bourgain and Dyatlov observed that $\delta$-regular sets on $\R$ admit damping functions as in Definition~\ref{def:damp} above with $\alpha=(1+\delta)/2$. They obtained these functions as a consequence of the Beurling-Malliavin theorem~\cite{BM}. However, one does not need the full strength of this theorem. To be more precise, in place of the original Beurling-Malliavin condition $\|(\log \omega)'\|_{\infty}<\infty$,   with $\omega$ the weight,  a much easier proof is possible (via outer functions) if we assume instead that $\|(H \log \omega)'\|_{\infty}\ll 1$ where $H$ is the Hilbert transform on~$\R$, see~\cite[Section 1.14, Theorem~1]{seven}.  By means of this technique, Jin and 
Zhang~\cite[Lemma~4.1]{JZ} proved the following quantitative result on damping functions. 
\begin{lemma}
\label{lem:JZ}
Let $Y\subset [-N,N]$ be $\delta_1$-regular on scales $2$ to $N$, with constant $C_{R}$, $0<\delta_1<1$. For any $0<c_{1}< 1$,  $Y$ admits a damping function with $\alpha=(1+\delta_1)/2$ and parameters $c_{1}$,
\EQ{\label{eq:c2c3}
c_{2}= \iota\, c_{1}^{6},\quad c_{3}= \iota\, c_{1} C_{R}^{-2}\delta_1(1-\delta_1),
}
where $\iota>0$ is some {{small}} absolute constant.  Instead of the pointwise global decay of $\langle \xi\rangle^{-1}$ in Definition~\ref{def:damp}, we have 
\EQ{
\label{eq:global c3decay}
|\what{\psi}(\xi)|\le \exp(-c_{3}\, \la \xi\ra^{\frac12})\quad \forall\; \xi\in\R
}
\end{lemma}

In this paper we need a slightly different version, where we have pointwise lower bound of $|\what{\psi}(\xi)|$ on $[-3/4,3/4]$.
The advantage of a pointwise lower bound over a $L^2$ bound is that it leads to a lower bound of the product of several $\what{\psi}$'s.

\begin{lemma}\label{lem:regularsetdamping}
Assume that $Y\subset [-N, N]$ is a $\delta_1$-regular set with constant $C_R$ on scales $2$ to $N$ and $\delta_1\in (0,1)$. 
Fix $0<c_1<1$, then there exists a function $\psi\in L^2(\R)$ such that 
\EQ{\nn
&\supp \psi\subset \left[-\frac{c_1}{10},\frac{c_1}{10}\right],\\
&|\what{\psi}(\xi)|\leq \exp(-c_3 \la \xi\ra^{1/2}),\ \ \forall \xi \in \R,\\
&|\what{\psi}(\xi)|\leq \exp(-c_3 \Theta(|\xi |) |\xi |),\ \ \forall \xi \in Y,\ |\xi |\geq 10,
}
and 
\EQ{\label{eq:psi c2}
|\what{\psi}(\xi)|\geq c_2,\ \ \forall \xi\in [-3/4,3/4],
}
with 
\[\alpha=(1+\delta_1)/2,\ \ c_2=\iota\, c_1^{10},\ \ c_3=\iota\, c_1 C_R^{-2} \delta_1(1-\delta_1),\]
where $\iota>0$ is some {{small}} absolute constant.
\end{lemma}
We include the proof of Lemma \ref{lem:regularsetdamping} in Appendix \ref{sec:app}.

In higher dimensions, we reduce ourselves to this one-dimensional setting by taking finite unions of products. For simplicity, we restrict ourselves to two dimensions, although the exact analogue can be done in any finite dimension. 

\begin{definition}
\label{def:Yadmissible} 
Pick some $\eps_{0}\in (0,1)$ and let $Y\subset [-N,N]^{2}$ with $N\ge10$ be of the form
\EQ{\label{eq:Ycover}
Y&\subset \bigcup_{j=1}^{m} Y_{j},\\
 Y_{j} & =  \{ \xi_{1}\vec e_{j,1} + \xi_{2}\vec e_{j,2}\::\: \xi_{i}\in Y_{j,i}, \; i=1,2\}
}
Here $\vec e_{j,i}\in \Sph^{1}$ with $ | \vec e_{j,1}\cdot \vec e_{j,2}|<1-\eps_{0}$ for all $1\le j\le m$, and $Y_{j,i}\subset [-2N,2N]$ are  $\delta_1$-regular on scales $2$ to $N$, with constant $C_{R}$, $0<\delta_1<1$.
In that case $Y$ is called {\em admissible at scale $N$} with parameters $\delta_1, C_{R},\eps_{0},m$.  In general dimensions, we require that $\vec e_{j,i}$ are unit vectors with $|\det(\vec e_{j,1},\ldots,\vec e_{j,d})|\ge \eps_0$, cf.~\eqref{eq:YYj}. 
\end{definition}

Throughout, we will freeze $\eps_{0}$ and constants are allowed to depend on it. 
These admissible sets carry damping functions. 

{{We note that for our proof of Theorem~\ref{thm:main1}, we only need $m=1$. 
We give a construction with arbitrary $m\geq 1$ here, since the construction itself may be of independent interest.}}

\begin{lemma}\label{lem:admissibledamping}
Let $Y\subset [-N,N]^{2}$ be admissible as in Definition~\ref{def:Yadmissible}. Then $Y$ admits a damping function with 
parameters $c_{1}$, 
\EQ{ \nn 
c_{2} &= \iota^{2m+4} c_{1}^{20m+4}\, m^{-20m}\, C_{R}^{-8}(\delta_1(1-\delta_1))^{4} \\
c_{3} &=\iota\,  c_{1}\,m^{-1} C_{R}^{-2}\delta_1(1-\delta_1) 
}
where $\iota>0$ {{is a small constant that}} depends on $\eps_{0}$. 
\end{lemma}

{{
\begin{rmk}\label{rmk:admissbledampingd}
For general dimension $d$, we can take
\EQ{\nn
c_{2} &=\iota^{m} c_{1}^{(10m+2)d}\, m^{-10md}\, C_{R}^{-4d}(\delta_1(1-\delta_1))^{2d} \\
c_{3} &=\iota\,  c_{1}\,m^{-1} C_{R}^{-2}\delta_1(1-\delta_1) 
}
where $\iota>0$ is a small constant that depends on $\eps_0$ and $d$.
\end{rmk}
}}

\begin{proof}
Let $\psi_{j,i}$ be the damping function associated with $Y_{j,i}$ via Lemma~\ref{lem:regularsetdamping} with parameters $\tilde c_{1}:= \eps_{1}c_{1} m^{-1}$ where $\eps_{1}$ is a small parameter depending on $\eps_{0}$, and $c_{2}, c_{3}$ as given by Lemma~\ref{lem:regularsetdamping}, but in terms of~$\tilde c_1$. I.e.,
\[
c_{2}= \iota\, \varepsilon_1^{10}\, c_{1}^{10}m^{-10},\quad c_{3}= c_{1}\, m^{-1} \iota\, \varepsilon_1\, C_{R}^{-2}\delta_1(1-\delta_1).
\] 
We will absorb the constant $\varepsilon_1$ into $\iota$, hence $\iota$ will depend on $\varepsilon_0$. 
In the following we will also allow $\iota$ to change its value from line to line, as long as it only depends on $\varepsilon_0$.

Denote the coordinates associated with the basis 
$\vec e_{j,1},\, \vec e_{j,2}$ by $(\xi_{j,1},\xi_{j,2})$. We set, with $\xi\in \R^{2}$,  
\[
\what{\psi}(\xi) := \prod_{j=1}^{m} \what{\psi_{j}}(\xi),\qquad \what{\psi_{j}}(\xi):=\what{\psi_{j,1}}(\xi_{j,1})  \what{\psi_{j,2}}(\xi_{j,2}) 
\]
Then 
\EQ{ 
|\what{\psi}_{j}(\xi)| &\le   \exp(-c_{3}\, \la \xi_{j,1}\ra^{\frac12})\exp(-c_{3}\, \la \xi_{j,2}\ra^{\frac12}) \\
&\le  \exp(-c_{3}\, \la \xi \ra^{\frac12}),
}
where $c_{3}$, more precisely, $\iota$, can change its value in the last line depending on $\eps_{0}$. 
Taking products gives
\EQ{\label{eq:psi global}
|\what{\psi}(\xi)|  \le \exp(-mc_{3}\, \la \xi \ra^{\frac12}) = \exp(- c_{1} \nu \, \la \xi \ra^{\frac12}), \quad \nu = \iota\, C_{R}^{-2}\delta_1(1-\delta_1)
}
In particular, $\psi\in L^{2}(\R^{2})$ as well as $\psi_{j}\in L^{2}(\R^{2})$.
Since $\psi_j$ are also compactly supported functions, $\psi_j\in L^1(\R^2)$.
Hence in the sense of $L^1$ functions, 
\[
\psi = \Asterisk_{j=1}^{m}\; \psi_{j}
\]
whence 
\EQ{\nn
\supp(\psi) \subset \sum_{j=1}^{m} \supp(\psi_{j})  \subset \sum_{j=1}^{m} [-c_{1}m^{-1}, c_{1}m^{-1}]^{2} \subset [-c_{1} , c_{1}]^{2},
}
where we used that each $\psi_{j,i}$ is a damping function with $\tilde{c}_1=\varepsilon_1 c_1m^{-1}$.
Next, if $\xi\in Y_{j}$, then 
\EQ{\nn 
|\what{\psi}_{j}(\xi)| &\le  \exp\big( -c_3 \Theta(|\xi_{j,1}|) |\xi_{j,1}|) \exp\big( -c_3 \Theta(|\xi_{j,2}|) |\xi_{j,2}|)  \\
&\le \exp\big( -c_3 \Theta(|\xi|_1) |\xi|_1) 
}
where again $\iota$ is allowed to change in the second line. 
Since $Y$ is covered by the union of~$Y_{j}$, we have
\EQ{\label{eq:Y global}
|\what{\psi}(\xi)|  \le \exp\big( -c_3 \Theta(|\xi|_1) |\xi|_1\big)  \quad\forall\;  \xi\in Y. 
}
Finally, from \eqref{eq:psi c2}, for each $1\le j\le m$,
\EQ{\nn
|\what{\psi_{j}}(\xi)|\ge c_{2}^{2},\ \ \forall \xi_{j,1}, \xi_{j,2}\in [-3/4,3/4].
}
Hence,
\[
\| \what{\psi}\|_{L^{2}([-1,1]^{2})} \ge c_2^{2m} |E|^{\frac12},
\]
where $E$ is the subset of $[-1,1]^{2}$ where all conditions $\xi_{j,i} \in [-3/4,3/4],\;\; i=1,2$, $1\le j\le m$, are met. Clearly, $|E|^{\frac12}$ is some number depending on~$\eps_{0}$.  It follows that 
\EQ{\label{eq:L2 lower}
\| \what{\psi}\|_{L^{2}([-1,1]^{2})} \ge \iota^{2m}\, c_{1}^{20m} m^{-20m}
}
where $\iota$ depends on $\eps_{0}$. 

We required $|\what{\psi}(\xi)|\le \la\xi\ra^{-2}$ in our definition of damping function, see Definition~\ref{def:damp}. 
Since for any $0<\rho<1$ 
\[
\exp\big(-\rho\, \la \xi\ra^{\frac12}\big)\le 5\rho^{-4} \la \xi\ra^{-2}
\]
It follows from \eqref{eq:psi global} that $\wtil\psi:=\frac15(c_{1}\nu)^{4}\psi$ is a damping function in the sense of the definition.  Since $\frac15(c_{1}\nu)^{4}\le 1$, the decay \eqref{eq:Y global} remains intact, as does the support condition. However, \eqref{eq:L2 lower} needs to be modified:
\EQ{\nn
\| \what{\wtil\psi}\|_{L^{2}([-1,1]^{2})} &\ge \frac15(c_{1}\nu)^{4}\iota^{2m}\, c_{1}^{20m} m^{-20m} \\
& = \frac15 \iota^{2m+4}  c_{1}^{20m+4} m^{-20m}  C_{R}^{-8}(\delta_1(1-\delta_1))^{4}.
}
Absorbing the $\frac15$ into $\iota$, the lemma is proved. 
\end{proof}

Finally, we need to check that $Y$ remains admissible if it is transformed by the similarities in \eqref{eq:Y changed}. 

\begin{lemma}
\label{lem:ad stable}
Let  $Y\subset [-N,N]^{d}$ with $N\ge10$ be admissible at scale $N$ with parameters $\delta_1, C_{R},\eps_{0},m$. Let $L\ge {{4}}$ be an integer. Then for all  integers $n\geq 0$ with $L^{n+1}\le N$ and for all 
$$\eta\in [-NL^{-n}-3, NL^{-n}+3]^{d}$$ the set 
\[
L^{-n}Y+[-4,4]^d+\eta
\]
is admissible at scale  $S(2NL^{-n}+7)$ with parameters $\delta_1, 576 S^2 C_{R},\eps_{0},m$, where $S=S(\eps_0)\ge1$. 
\end{lemma}
\begin{proof}
First, 
\[
L^{-n}Y+[-4,4]^d+\eta \subset  [-2NL^{-n}-7, 2NL^{-n}+7]^d
\]
for all $\eta$ as above. Second, by \eqref{eq:Ycover}, 
\EQ{\nn 
&L^{-n}Y+[-4,4]^d+\eta \subset \bigcup_{j=1}^{m} \Big( L^{-n} Y_{j} +[-4,4]^d+\eta \Big) ,\\
&L^{-n} Y_{j}  =  \Big\{\sum_{k=1}^d \xi_{k}\vec e_{j,k} \::\: \xi_{k}\in L^{-n}Y_{j,k}, \; k=1,2,...,d.\Big\}
}
and
\[
L^{-n} Y_{j} +[-4,4]^d+\eta \subset \Big\{ \sum_{k=1}^d \xi_{k}\vec e_{j,k} \::\: \xi_{k}\in L^{-n}Y_{j,k} + [-4S,4S]+ \eta_{j,k}, \; k=1,2,...,d.\Big\}
\]
where $S=S(\eps_0)\ge 1$ and  $|\eta_{j,k}|\le S(NL^{-n}+3)$.   
By Lemmas 2.1, 2.2, 2.3 in \cite{BD}, see also Lemmas \ref{lem:BDlemma21}, \ref{lem:BDlemma22}, \ref{lem:BDlemma23} with $d=1$, 
the sets
\[
L^{-n}Y_{j,k} + [-4S,4S]+ \eta_{j,k} \subset [-S(2NL^{-n}+7), S(2NL^{-n}+7)] 
\]
are $\delta_1$-regular with constant $576 S^2 C_R$ on scales $2$ to $S(2NL^{-n}+7)$.  
Indeed, \linebreak 
for $n\geq 1$,
Lemma \ref{lem:BDlemma21} implies that $L^{-n}Y_{j,k}$ is $\delta_1$-regular from scales $2L^{-n}\leq 1/2$ to $L^{-n}N$ with constant $C_R$.
Lemma \ref{lem:BDlemma23} implies that 
\[L^{-n}Y_{j,k}+[-4S,4S]=L^{-n}Y_{j,k}+8S[-1/2, 1/2]\] is $\delta_1$-regular from scales $1$ to $L^{-n}N$ with constant $32SC_R$.
Lemma \ref{lem:BDlemma22} allows us to increase the upper scale from $L^{-n}N$ to $9S L^{-n}N\geq S(2L^{-n}N+7)$, with changing the constant from $32S C_R$ to $576S^2 C_R$.
Note that shifting a set does not change its $\delta_1$-regularity, hence $L^{-n}Y_{j,k} + [-4S,4S]+ \eta_{j,k}$ is $\delta_1$-regular with constant $576S^2 C_R$.
The proof for $n=0$ is similar.

The lemma now follows from Definition~\ref{def:Yadmissible}. 
\end{proof}

\subsection{Proof of Theorem~\ref{thm:main1}}
\begin{proof}
The proof of Theorem~\ref{thm:main1} is now a corollary to Theorem~\ref{thm:FUP2} and the considerations in this section, {{with $m=1$}}. 
We will keep track of various constants in order to obtain the effective exponent $\beta$.

First, let 
\[L:=\lceil (2\sqrt{5}C_R)^{\frac{2}{2-\delta}} \rceil {{\geq 4}},\]
be as in \eqref{def:L}.
Lemma \ref{lem:BDlemma210} implies that for all $n\geq 0$ with $L^{n+1}\leq N$, $X$ is porous at scale $L$ with depth $n$.
This verifies the porosity condition on $X$ in Theorem \ref{thm:FUP2}.

Combining Lemma \ref{lem:admissibledamping}, {{more specifically Remark~\ref{rmk:admissbledampingd}}}, with \ref{lem:ad stable}, we obtain that 
for any $n\in \N$ such that $L^{n+1}\leq N$, and for all $\eta\in [-L^{-n}N-3, L^{-n}N+3]^d$, the set 
\[L^{-n}Y+[-4,4]^d+\eta\]
admits a damping function with parameters $c_1$,
\EQ{\nn
c_2&=\iota c_1^{12 d} (576 S^2 C_R)^{-4d} (\delta_1(1-\delta_1))^{2d},\\
c_3&=\iota c_1 (576 S^2 C_R)^{-2} \delta_1(1-\delta_1),
}
where $\iota$ and $S$ are constants depending on $\varepsilon_0$. We absorb the constant $S$ into $\iota$, and allow $\iota$ to depend on $d$ as well. Hence we can simply write
\EQ{\nn
c_2&=\iota c_1^{12d} C_R^{-4d} (\delta_1(1-\delta_1))^{2d},\\
c_3&=\iota c_1 C_R^{-2} \delta_1(1-\delta_1).
}
Note that this verifies the condition on $Y$ in Theorem \ref{thm:FUP2}.

Before applying Theorem \ref{thm:FUP2}, let us first determine the constant $C_*$ in \linebreak 
Corollary \ref{cor:UPY} with $c_1, c_2, c_3$ defined above.
Recall that 
\[C_*=e^{\frac{c_3 \Theta(R_1) (R_1+2)}{2}},\]
with $\alpha=(1+\delta_1)/2$ and
\EQ{\label{eq:defR1}
R_1=
\begin{cases}
\exp\Big[\Big(\frac{C_R^2}{\iota c_1 \delta_1(1-\delta_1)}\Big)^{\frac{2}{1-\delta_1}}\Big]\\
\exp\Big(4^{\frac{2}{1-\delta_1}}\Big)\\
\Big(\frac{C_R^2 (-\log{c_1})^d}{\iota c_1 \delta_1(1-\delta_1)}\Big)^8\\
\Big[4\log\Big(\frac{C_R^{8d}}{\iota c_1^{24 d} (\delta_1(1-\delta_1))^{4d}}\Big)\Big]^2\\
(8d)^4\\
\frac{C_R^4}{\iota c_1^2 (\delta_1(1-\delta_1))^2}
\end{cases}
}
be as in \eqref{def:R}, in which we absorb all the $d$-dependent constants into $\iota$. 

Now we can apply Theorem \ref{thm:FUP2} with 
\[c_1=(2L)^{-1}=\Big(2 \lceil (2\sqrt{5}C_R)^{\frac{2}{2-\delta}} \rceil \Big)^{-1}.\]
We need to trace out the constant $\beta$.

Plugging $c_1$ into \eqref{eq:defR1}, and making $\iota$ smaller if necessary (depending only on $d$ and $\varepsilon_0$), we have
\[R_1\leq \exp\Big[\Big(\frac{C_R^2}{\iota \delta_1(1-\delta_1)}\Big)^{\frac{6-2\delta}{(1-\delta_1)(2-\delta)}}\Big]=:R_2.\]
This implies 
\[C_*=\exp\Big(c_1 C_R^{-2} \delta_1(1-\delta_1) \Theta(R_1)(R_1+2)\Big)\leq \exp(R_2).\]

Recall $T_0$ as in \eqref{def:T} and $\gamma_0$ as in \eqref{def:gamma0}.
We compute that,
\EQ{\label{eq:T1}
T_0
=\Bigg\lceil \frac{\log(2C_{\phi}C_*^2+\sqrt{4C_{\phi}^2 C_*^4 +c_{\phi}^2})}{\log L} \Bigg\rceil
\leq &\frac{2\log{C_*}+\log{(5C_{\phi})}}{\log L}\\
\leq &\frac{2R_2+\log{(5C_{\phi})}}{\log L}=:T_1,
}
and
\EQ{\label{eq:gamma0T1}
\gamma_0(T_1)=\frac{1-{c_{\phi}^2}/{L^{2(T_1-1)}}}{2C_*^2}\geq \frac{1}{4C_*^2}\geq \frac{1}{4}\exp(-2R_2).
}
In both inequalities above, we used $C_*\leq \exp(R_2)$.

Recall $\beta$ as in \eqref{eq:Xf}.
Use that $-\log(1-x)\geq x$ for $x<1$, we have
\EQ{\nn
\beta=-\frac{\log(1-\gamma_0(T_1)/2)}{T_1\log L}
\geq \frac{\gamma_0(T_1)}{2T_1 \log L}.
}
Combining this with the estimates of $T_1$ and $\gamma_0(T_1)$ as in \eqref{eq:T1} and \eqref{eq:gamma0T1}, we have
\EQ{\nn
\beta\geq \exp\Big\{-\exp\Big[\Big(\frac{C_R^2}{\iota \delta_1(1-\delta_1)}\Big)^{\frac{6-2\delta}{(1-\delta_1)(2-\delta)}}\Big]\Big\},
}
with $\iota$ being a small constant depending on $\varepsilon_0$ and $d$.
This finishes the proof.
\end{proof}

{{Corollary~\ref{thm:main2} follows from Theorem~\ref{thm:main1} by the triangle inequality.
\begin{rmk}
If we try to combine the construction of a damping function for $m$ covers as in Lemma~\ref{lem:admissibledamping}, with Theorem~\ref{thm:FUP2},
we could allow $m$ to grow in $N$ like $\log\, \log\, \log N$.
This is worse than the power law growth obtained via the triangle inequality.
\end{rmk}
}}

\subsection{Distortion of $Y$ by diffeomorphisms}\label{sec:distort}
 
 In this section we prove Theorem~\ref{thm:main3}. We need to show that Theorem~\ref{thm:main1} remains valid if an admissible set $Y$ is distorted by a diffeomorphism $\Phi_N(x) = N\Psi_0(x/N)$ from the cube 
$[-N,N]^d\to [-N,N]^d$, cf.~\eqref{eq:PhiN}. The argument is related to Section~4 of~\cite{BD}. 
Thus, let  $Y=\Phi_N(\wtild Y)$ where $\wtild Y\subset [-N,N]^d$ is admissible at scale $N$ with parameters 
$\delta_1, C_R, \varepsilon_0, {{m=1}}$ in the sense of Definition~\ref{def:Yadmissible}.  
Suppose $f\in L^2(\R^d)$ with $\supp(\what{f})\subset Y$ and set $\what{g}:=\what{f}\circ \Phi_N$ so that $\supp(\what{g})\subset \wtild Y$.  Furthermore, 
\EQ{
\label{eq:FTtwist}
f(x) &=  \int_{[-N,N]^d}  e^{2\pi i x\cdot  \xi} \, \what{f}(\xi)\, d\xi = \int_{[-N,N]^d}  e^{2\pi i x\cdot  \xi} \, \what{g}(\Phi_N^{-1}(\xi))\, d\xi\\
& = \int_{[-N,N]^d}  e^{2\pi i x\cdot  \Phi_N(\eta)} \, \what{g}(\eta)\ |\det(D\Phi_N(\eta))|\, d\eta
}
We claim that for some $\beta>0$ and $C>0$ depending on all the same parameters in Theorem~\ref{thm:main1} as well as on~$D_0$,  
\EQ{
\label{eq:main claim} 
\Big\|  \int_{[-N,N]^d}  e^{2\pi i x\cdot  \Phi_N(\eta)} \, \what{h}(\eta) \, d\eta  \Big\|_{L^2(X)} \le C N^{-\beta} \|h\|_2
}
for all $h\in L^2$ with $\supp(\what{h})\subset\wtild Y\subset [-N,N]^d$, an admissible set in the sense of Definition~\ref{def:Yadmissible}.  
Setting $\what{h}(\eta):= \what{g}(\eta)\ |\det(D\Phi_N(\eta))|$, we  conclude from~\eqref{eq:main claim} that
\[
\|f\|_{L^2(X)} \le CN^{-\beta}\|\what{h}\|_2\le C N^{-\beta}\|\what{f}\|_2 = C N^{-\beta}\| {f}\|_2
\]
with possibly a different constant. So it remains to prove the claim~\eqref{eq:main claim}. 
We will prove it from another statement, namely 
\EQ{
\label{eq:main claim2} 
\Big\|  \int_{[-N,N]^d}  e^{2\pi i x\cdot  \Phi_N(\eta)} \,\one_Y(\eta) \,  h(\eta) \, d\eta  \Big\|_{L^2(X)} \le C N^{-\beta} \|h\|_2
}
for all $h\in L^2$. 
Notice that by Plancherel we could remove the Fourier transform from~$h$. 

\medskip

 To prove \eqref{eq:main claim2}, divide  $[-N,N]^d=\bigcup_k Q_k$ into congruent cubes of side length $L_N$ with $\frac12\sqrt{N} \le L_N \le \sqrt{N}$.  Let $\{\chi_k\}_k$ be a partition of unity adapted to these cubes. 
With $\eta_k$ being the center of~$Q_k$, 
\EQ{\label{eq:Tk}
& \int_{[-N,N]^d}  e^{2\pi i x\cdot  \Phi_N(\eta)} \, \one_Y(\eta) \,  {h}(\eta) \, d\eta \\ 
&= \sum_k \int_{\R^d}  e^{2\pi i x\cdot  \Phi_N(\eta)} \, \chi_k(\eta)\, \one_Y(\eta) \,  {h}(\eta) \, d\eta  \\
& = \sum_k \int_{\R^d}  e^{2\pi i x\cdot  (\Phi_N(\eta_k)+ D\Phi_N(\eta_k) (\eta-\eta_k) )} \, a_k(x,\eta) \,  \one_Y(\eta) \, {h}(\eta) \, d\eta     =: \sum_k (T_k h)(x)
}
where 
\EQ{\label{eq:Taylor} 
a_k(x,\eta) &:= e^{2\pi i x\cdot R_k(\eta)} \, \chi_k(\eta) \\
R_k(\eta) &:= \int_0^1 (1-t) \langle D^2 \Phi_N(\eta_k + t(\eta-\eta_k))(\eta-\eta_k), \eta-\eta_k\rangle \, dt
}
the latter being the error in the second order Taylor expansion (we are suppressing the parameter $N$ here). Then 
\EQ{\label{eq:ak bounds}
\| R_k\|_{L^\infty(Q_k)} &\le C=C(d,D_0) \\
  \| \partial_x^\alpha \partial_\eta^\gamma\, a_k(x,\eta)\|_{L^\infty([-1,1]^d\times Q_k)} &\le C(d,D_0,\alpha,\gamma) N^{-\frac{|\gamma|}{2}}
}
for every multi-indices $\alpha, \gamma$.  By H\"{o}rmander's variable coefficient Plancherel theorem, 
\EQ{\label{eq:Hor}
\max_k \| T_k \|_{2\to2} \le C(d,D_0) 
}
This follows by the usual $T^*T$ argument: 
\EQ{\label{eq:TkKk}
\| T_k h\|_2^2 &= \langle T_k^* T_k h, {h}\rangle \\
(T_k^* T_k h) (\eta') &= \int_{\R^d} K_k(\eta',\eta)\,  {h}(\eta)\, d\eta \\
K_k(\eta',\eta) &= \int_{\R^d}  e^{2\pi i x\cdot  (\Phi_N(\eta)-\Phi_N(\eta'))} \, \one_Y(\eta) \, \one_Y(\eta') \,\chi_k(\eta) \chi_k(\eta') \, dx
}
Since $\| \Phi_N(\eta)-\Phi_N(\eta') \|\ge D_0^{-1} \|\eta-\eta'\|$ in the sense of Euclidean lengths, repeated integrations by parts yield the decay
\[
|K_k(\eta',\eta)| \le C(d,D_0) \langle \eta-\eta'\rangle^{-d-1}
\]
whence \eqref{eq:Hor} follows by Schur's test.  In particular, $\|\one_X T_k\|_{2\to2} \le C$ with the same constant as in~\eqref{eq:Hor}. 

Next, we would like to show that $\one_X T_k$ and $\one_X T_\ell$ do not interact much  for all cubes $Q_k,Q_\ell$ which are not nearest neighbors.
In order to integrate by parts in~$x$, cf.~\eqref{eq:TkKk}, we need to smooth out $\one_X$ at the correct scale. Define $$X(N^{-\frac{1}{2}}):=X+[-N^{-\frac{1}{2}}, N^{-\frac{1}{2}}]^d.$$
By~\cite[Lemma~3.3]{DyZ} there exists a smooth $\psi$ taking values in $[0,1]$ with $\psi=1$ on $X $ and with $\supp(\psi)\subset X(N^{-\frac{1}{2}})$, as well as so that 
\EQ{\label{eq:psi bounds}
\| \partial_x^\alpha \psi\|_\infty \le C(\alpha) N^{\frac{|\alpha|}{2}}
}
for all multi-indices.  Define $S_k:=\psi \, T_k$. On the one hand, $S_k$ still obeys~\eqref{eq:Hor}. On the other hand, for any cubes 
$Q_k,Q_\ell$ which are not nearest neighbors one has 
\EQ{
\label{eq:k ell}
\| S_k^* S_\ell\|_{2\to2}  & \le C_p(d,D_0)\, N^{\frac{p}{2}}\,\dist(Q_k,Q_\ell)^{-p}
}
for every positive integer $p$. 
This follows from the fact that the kernel of $S_k^* S_\ell$ equals 
\[
K_{k,\ell}(\eta',\eta) = \int_{\R^d}  e^{2\pi i x\cdot  (\Phi_N(\eta)-\Phi_N(\eta'))} \, \one_Y(\eta) \, \one_Y(\eta') \,\chi_k(\eta) \chi_\ell(\eta') \,\psi(x)^2\,  dx
\]
Using the differential operator 
\[
\calL = \frac{1}{2\pi i} \frac{\Phi_N(\eta)-\Phi_N(\eta')}{\|\Phi_N(\eta)-\Phi_N(\eta')\|^2}\cdot   \nabla_x, 
\]
which obeys 
\[
\calL\,   e^{2\pi i x\cdot  (\Phi_N(\eta)-\Phi_N(\eta'))} =  e^{2\pi i x\cdot  (\Phi_N(\eta)-\Phi_N(\eta'))},
\]
repeated integration by parts now yields~\eqref{eq:k ell}. 
Finally, given any $k$, only a uniformly bounded number of choices of $\ell$ will satisfy 
\[
S_k S_\ell^* = \psi \, T_k T_\ell^* \, \psi =0
\]
This is due to the fact that $\chi_k(\eta) \chi_\ell(\eta)=0$ up to a bounded number of choices of $\ell$ given~$k$.  
If we label the cubes by lattice points $\uk\in\Z^d$, then $\eta_{\uk}=L_N\uk$ whence 
\[
N^{\frac{p}{2}}\,\dist(Q_{\uk},Q_{\uell})^{-p} \les  N^{\frac{p}{2}} (L_N|\uk-\uell|)^{-p} \les |\uk-\uell|^{-p}
\]
which is summable over $\Z^d$ provided $p>d$. On the other hand, we also have 
\[
\| S_k^* S_\ell\|_{2\to2} \le \|S_k\|_{2\to2} \|   S_\ell \|_{2\to2}  \le B^2, \quad B:=\sup_j \| S_j\|_{2\to2} 
\]
Combining these two estimates we infer that for any $0<\eps<1$, 
\[
\|S_{\uk} S_{\uell}^*\|_{2\to2}+ \| S_{\uk}^* S_{\uell}\|_{2\to2} \le {{C(d,D_0,\eps)}} \, B^{2(1-\eps)} \langle \uk-\uell \rangle^{-2(d+1)} 
\]
for all $\uk,\uell\in\Z^d$. 
Note that $B\le  C(d,D_0) $ by H\"{o}rmander's bound~\eqref{eq:Hor}. 
Hence by Cotlar's lemma, 
\EQ{\label{eq:cotlar}
\Big\|  \int_{[-N,N]^d}  e^{2\pi i x\cdot  \Phi_N(\eta)} \, \one_Y(\eta) \,  {h}(\eta) \, d\eta  \Big\|_{L^2(X)} \le C(\eps,d,D_0) \, \max_k \|S_k\|_{2\to2}^{1-\eps}.
}

\medskip

The claim~\eqref{eq:main claim2} will now follow from \eqref{eq:cotlar} by applying the fractal uncertainty principle of Theorem~\ref{thm:main1} to each $S_k$. For this we need to linearize the phase as in~\eqref{eq:Tk}  which in turn makes the localization to scales~$\sqrt{N}$ necessary. 

To be specific, we  reduce~\eqref{eq:main claim2} to the following estimate. 
Let $\psi_0$ be compactly supported functions satisfying the bounds 
\EQ{\label{eq:psi1psi2}
 \| \partial_x^\alpha \psi_0\|_\infty \le C_s\, N^s \qquad \forall\; |\alpha|=s\ge0
}
where $N\ge1$ is arbitrary and all constant are independent of~$N$. 
We assume that $\psi_0$ is supported in a $\delta$-regular set in $[-1,1]^d$ on scales $1/N$ to~$1$, and with $0<\delta<d$. 
Let 
$$Z=N^{-1} Y_1\subset[-1,1]^d$$ 
be a rescaled version of an admissible set~$Y_1$ at scale $N$ as in Definition~\ref{def:Yadmissible}. 
Furthermore, suppose that the symbol $a$ is smooth and compactly supported with the bounds 
\EQ{\label{eq:abds}
\| \partial_x^\alpha \partial_\xi^\gamma a(x,\xi) \|_\infty \le C(\alpha,\gamma)\qquad \forall\;\alpha,\gamma
}
Then for some $\beta>0$ and $C$ as above, 
\EQ{
\label{eq:main claim3}
\Big\| N^{\frac{d}{2}}\int_{\R^d}  e^{2\pi i N x\cdot  \xi} \; \psi_0(x) a(x,\xi) \one_Z(\xi) h(\xi)\, d\xi\Big\|_2 &\le C N^{-\beta}\|h\|_2
}
Indeed, 
\EQ{\nn 
& \Big\| \int_{\R^d}  e^{2\pi i x\cdot  (\Phi_N(\eta_k)+ D\Phi_N(\eta_k) (\eta-\eta_k) )} \; \psi(x)\,  a_k(x,\eta) \,  \one_Y(\eta) \, {h}(\eta) \, d\eta  \Big\|_{2} \\
& \les \Big\| \int_{\R^d}  e^{2\pi i x\cdot  \zeta } \;  \psi(x)\, a_k(x,D\Phi_N(\eta_k)^{-1}\zeta) \,  \one_Y(D\Phi_N(\eta_k)^{-1}\zeta) \, 
{h}(D\Phi_N(\eta_k)^{-1}\zeta) \, d\zeta  \Big\|_{2}\\
& = N^{\frac{d}{4}} \Big\| \int_{\R^d}  e^{2\pi i N^{\frac12} x\cdot  \xi } \;  \psi(x)\, \wtild a_k(x,N^{\frac12} \xi) \,  N^{\frac{d}{4}} 
\one_{Y_1}(N^{\frac12} \xi ) \, \wtild{h}(N^{\frac12} \xi) \, d\xi  \Big\|_{2}.
}
Here $\wtild a$, $\wtild{h}$ signify the functions on the second line but with the linear isomorphism $D\Phi_N(\eta_k)^{-1}$ included, and 
$Y_1=D\Phi_N(\eta_k) Y$ is an admissible set at scale $C(D_0) N$ with constant $C(D_0)$ and parameters that depend on $D_0$. 
{{We could cover $Y_1$ by admissible sets at scale $N$ and use triangle inequalities.
Hence without loss of generality, we assume $Y_1$ is admissible at scale $N$.}}
Note that all these only affect the estimates through $D_0$-dependent constants. 
Note that $\one_{Y_1}(N^{\frac12}\xi)=\one_Z(\xi)$, with $Z=N^{-\frac{1}{2}}Y_1$.
By~\eqref{eq:psi bounds}, $\psi_0(x):=\psi(x)$ satisfies the required bound, furthermore $\psi_0$ is supported on a $X(N^{-\frac{1}{2}})$ which is a $\delta$ regular set at scales $N^{-\frac{1}{2}}$ to $1$, see Lemma \ref{lem:BDlemma23}.
As for the amplitude, and ignoring the distinction between $\wtild a_k$ and $a_k$, 
\EQ{\nn
a_k(x, N^{\frac12}\xi) &:= e^{2\pi i x\cdot R_k(N^{\frac12} \xi)} \, \chi_k(N^{\frac12} \xi ) \\
R_k(N^{\frac12}\xi) &:= N\int_0^1 (1-t)  \langle D^2 \Phi_N(\eta_k + t(N^{\frac12}\xi-\eta_k))(\xi-\eta_k'), \xi-\eta_k'\rangle \, dt
}
where $\eta_k' = N^{-\frac12}\eta_k$. Setting $a(x,\xi) = a_k(x, N^{\frac12}\xi)$, we conclude from \eqref{eq:ak bounds} that all derivatives of $a$ are bounded, uniformly in~$N$. 
Finally, $$ \| N^{\frac{d}{4}}\wtild{h}(N^{\frac12} \xi) \|_2 \simeq  \|h\|_2.$$ 
Thus, we can apply~\eqref{eq:main claim3} with $N$ replaced by $N^{\frac12}$ to obtain a gain of $N^{-\beta/2}$, and we are done. 

It remains to prove~\eqref{eq:main claim3}. If $a=1$ on the $\supp(\psi_0)\times Y$, then this follows immediately from Theorem~\ref{thm:main1} by a rescaling. Indeed, one has by that theorem
\EQ{\nn 
& \Big\| N^{\frac{d}{2}}\int_{\R^d}  e^{2\pi i N x\cdot  \xi} \; \psi_0(x)\,   \one_Z(\xi)\, h(\xi)\, d\xi\Big\|_2 \\ 
& =  \Big\|  \int_{\R^d}  e^{2\pi i  x\cdot  \xi} \; \psi_0(x)\,    \one_Z(\xi/N)\, N^{-\frac{d}{2}} h(\xi/N)\, d\xi\Big\|_2   \\
&\les N^{-\beta} \| N^{-\frac{d}{2}} h(\xi/N) \|_2 = N^{-\beta} \|   h  \|_2. 
}

Let us denote
\[(Ah)(x):=N^{\frac{d}{2}}\int_{\R^d}  e^{2\pi i N x\cdot  \xi} \; a(x,\xi)  h(\xi)\, d\xi.\]
Thus it suffices to show that
\[\|\psi_0\, A\,  \one_{Z}\|_{2\to 2}\leq CN^{-\beta}.\]
Let $\rho\in (0,1)$ with its valued determined later.

First let us note that by the usual $A^*A$ argument, we have H\"ormander's bound, 
\EQ{\label{eq:Anorm}
\|A\|_{2\to 2}\leq C.
}
Next we decompose $\psi_0\, A\,  \one_{Z}$ into the following
\begin{align*}
\psi_0\, A\,  \one_{Z}=\psi_0\, \mathcal{F}_N^{-1}\,  A_1+ A_2\, \mathcal{F}_N\, A\, \one_{Z}&,\\
A_1:=\one_{\R^d\setminus Z(N^{-\rho})}\, \mathcal{F}_N\, A\, \one_{Z},\ \ \ A_2:=&\psi_0\, \mathcal{F}_N^{-1}\, 
\one_{Z(N^{-\rho})},
\end{align*}
in which $\mathcal{F}_N$ is the unitary semiclassical Fourier transform
\[\mathcal{F}_N f(\xi)=N^{\frac{d}{2}} \int_{\R^d} e^{-2\pi i N x \cdot \xi} f(x)\, dx=N^{\frac{d}{2}}\hat{f}(N \xi) .\]
Clearly, by \eqref{eq:Anorm}, we have
\EQ{\label{eq:Anorm1}
\|\psi_0\, A\, \one_{Z}\|_{2\to 2}\lesssim \|A_1\|_{2\to 2}+ \|A_2\|_{2\to 2}.
}
Thus it suffices to bound $\|A_1\|_{2\to 2}$ and $\|A_2\|_{2\to 2}$.

We compute the integral kernel of $A_1$:
\[K_{A_1}(\xi, \eta)=\one_{\R^d\setminus Z(N^{-\rho})}(\xi)\  \one_{Z}(\eta)\ N^d \int_{\R^d} e^{2\pi i N x\cdot (\eta-\xi)} a(x, \eta)\, dx.\]
Note that the Euclidean distance $\|\xi-\eta \|\geq N^{-\rho}$ on the support of $K_{A_1}$. Hence by repeated integration by parts in $x$, we obtain that 
\[|K_{A_1}(\xi, \eta)|\leq C_{d,\rho} N^{d-\lceil \frac{d+10}{1-\rho}\rceil} \la \eta-\xi\ra^{-\lceil \frac{d+10}{1-\rho} \rceil}\leq C_{d,\rho} N^{-10}.\]
By Schur's test, we arrive at 
\EQ{\label{eq:A1norm}
\|A_1\|_{2\to 2}\leq C N^{-10}.
}

In view of $A_2$.
Note that 
\[Z(N^{-\rho}) \subset N^{-1} Y_1(N^{1-\rho}),\]
where $Y_1(N^{1-\rho})$ can be covered by the union of $Y_{1,j}:=Y_1(1)+2j$, $j\in \Z^d$, $\|j\|_{\infty}\leq N^{1-\rho}/2$.
Each $Y_{1,j}$ is a admissible set at scale $N$.
Thus by a rescaling of Theorem \ref{thm:main1} and triangle inequality, we have for $f\in L^2(\R^d)$,
\EQ{\nn
\|A_2 f\|
=&\sum_{\|j\|\leq N^{1-\rho}/2} \Big\| N^{\frac{d}{2}} \psi_0(x) \int_{\R^d} e^{2\pi i N x\cdot \xi}\ \one_{N^{-1}Y_{1,j}}(\xi) f(\xi)\, d\xi \Big\|_2\\
\les & \sum_{\|j\|\leq N^{1-\rho}/2} \Big\| N^{-\frac{d}{2}} \psi_0(x) \int_{\R^d} e^{2\pi i  x\cdot \eta}\ \one_{Y_{1,j}}(\eta) f(\eta/N)\, d\eta \Big\|_2\\
\les & N^{-\beta+d(1-\rho)} \| N^{-\frac{d}{2}} f(\eta/N)\|_2
=CN^{-\beta+d(1-\rho)}\|f\|_2.
}
Hence for $\rho=1-\beta/2d$, 
\EQ{\label{eq:A2norm}
\|A_2\|_{2\to 2}\leq C N^{-\frac{\beta}{2}}.
}
Combining \eqref{eq:Anorm1}, \eqref{eq:A1norm} with \eqref{eq:A2norm}, we obtain \eqref{eq:main claim3}.
This concludes the proof of Theorem~\ref{thm:main3}. 

\appendix

\section{Regular sets}\label{sec:appregular}
We show that certain operations preserve the class of $\delta$-regular sets if we allow to increase the regularity constant and shrink the scales. 
\begin{lemma}\label{lem:BDlemma21}
Let $X$ be a $\delta$-regular set with $\delta\in (0,d)$ and constant $C_R$, on scales $\alpha_0$ to $\alpha_1$. 
Fix $\lambda>0$ and $y\in \R^d$. Then the set $\tilde{X}:=y+\lambda X$ is a $\delta$-regular set with constant $C_R$ on scales $\lambda \alpha_0$ to $\lambda \alpha_1$.
\end{lemma}
\begin{proof}
Taking the measure
\[\mu_{\tilde{X}}(A):=\lambda^{\delta}\mu_X (\lambda^{-1} (A-y)),\]
it is easy to verify.
\end{proof}

\begin{lemma}\label{lem:BDlemma22}
Let $X$ be a $\delta$-regular set with constant $C_R$ on scales $\alpha_0$ to $\alpha_1$. 
Fix $T>1$. Then $X$ is $\delta$-regular with constant $\tilde{C}_R:=2T^d C_R$ on scales $\alpha_0$ to $T\alpha_1$.
\end{lemma}
\begin{proof}
Let $I$ be a cube such that $\alpha_0\leq r_I \leq T\alpha_1$.
For $\alpha_0\leq r_I \leq \alpha_1$, the upper bound is immediate. 
For $\alpha_1< r_I \leq T\alpha_1$, $I$ can be covered by $\lceil T \rceil^d\leq 2 T^d$ cubes of side length $\alpha_1$ each, therefore
\[\mu_X(I)\leq 2T^d C_R \alpha_1^{\delta}\leq \tilde{C}_R r_I^{\delta}.\]

In view of the lower bound estimate, we assume $I$ is centered at a point in $X$. As before, we may assume $\alpha_1<r_I\leq T\alpha_1$. 
Let $I'\subset I$ be the cube with the same center and $r_{I'}=\alpha_1$.
Then 
\[\mu_X(I)\geq \mu_X(I')\geq C_R^{-1} \alpha_1^\delta\geq \tilde{C}_R^{-1} r_I^{\delta},\]
as claimed.
\end{proof}

\begin{lemma}\label{lem:BDlemma23}
Let $X$ be a $\delta$-regular set with constant $C_R$ on scales $\alpha_0$ to $\alpha_1$.
Fix $T\geq 1$. 
\begin{enumerate}
\item Suppose $\alpha_1\geq 2\alpha_0$,
the neighborhood $X+[-T\alpha_0, T\alpha_0]^d$ is $\delta$-regular with constant $\tilde{C}_R:=4^d T^d C_R$ on scales $2\alpha_0$ to $\alpha_1$.
\item Suppose that $\alpha_1\geq T\alpha_0$, 
then $X+[-T \alpha_0, T\alpha_0]^d$ is $\delta$-regular with constant $C_R'=4^d C_R$ on scales $T\alpha_0$ to $\alpha_1$.
\end{enumerate}
\end{lemma}
\begin{proof}
Let $\tilde{X}:=X+[-T\alpha_0, T\alpha_0]^d$ and define $\mu_{\tilde{X}}$ supported on $\tilde{X}$ by convolution
\[\mu_{\tilde{X}}(A):=\frac{1}{(T\alpha_0)^d}\int_{[-T\alpha_0, T \alpha_0]^d} \mu_X(A+y)\, dy.\]
Let $I$ be a cube such that $M \alpha_0\leq r_I\leq \alpha_1$ with $M\geq 1$. 
Then
\[\mu_{\tilde{X}}(I)\leq 2^d C_R r_I^{\delta},\]
which proves the upper bound estimates for both cases.

Now assume that $I$ is centered at a point $x_1\in \tilde{X}$. Take $x_0\in X$ such that $x_0\in x_1+[-T\alpha_0, T\alpha_0]^d$, 
and $I'$ be the cube centered at $x_0$ with side length $r_{I'}=r_I/2$.
Then 
\[\mu_X(I')\geq C_R^{-1} (r_I/2)^\delta\geq 2^{-d} C_R^{-1} r_I^{\delta}.\]

Let $J=x_0-x_1+[- \alpha_0/2,  \alpha_0/2]^d$, then $J\cap [-T\alpha_0, T\alpha_0]^d$ contains a cube with side length at least $\alpha_0/2$.
Clearly, $I'\subset I+y$ for any $y\in J$. Hence
\[\mu_{\tilde{X}}(I)\geq \frac{1}{(2T)^d} \mu_X(I')\geq \tilde{C}_R^{-1} r_I^{\delta},\]
which proves the lower bound estimate for (1).

Let $J=x_0-x_1+[- T\alpha_0/2, T\alpha_0/2]^d$, then $J\cap [-T\alpha_0, T\alpha_0]^d$ contains a cube with side length at least $T \alpha_0/2$.
Clearly, $I'\subset I+y$ for any $y\in J$. Hence
\[\mu_{\tilde{X}}(I)\geq \frac{1}{2^d} \mu_X(I')\geq (C_R')^{-1} r_I^{\delta},\]
which proves the lower bound estimate for (2).
\end{proof}

\begin{lemma}\label{lem:BDlemma26}
Let $X$ be a $\delta$-regular set with constant $C_R$ on scales $\alpha_0$ to $\alpha_1$, and $0<\delta<d$. Fix an integer 
\EQ{\label{def:L}
L\geq (2^{d/2}\sqrt{2d+1}C_R)^{\frac{2}{d-\delta}}.
}
Assume that $I$ is a cube with $\alpha_0\leq r_I/L\leq r_I\leq \alpha_1$ and $I_1,...,I_{L^d}$ is the partition of $I$ into cubes of side length $r_I/L$. 
Then there exists $\ell$ such that $X\cap I_{\ell}=\emptyset$.
\end{lemma}
\begin{proof}
Using Lemma \ref{lem:BDlemma21}, it suffices to consider $I=[0,L]^d$, $\alpha_0\leq 1\leq L\leq \alpha_1$.
We argue by contradiction. Assume that each $I_{\ell}$ intersects $X$. 
Then 
$I'_{\ell}:=I_{\ell}+[-1/2,1/2]^d$ contains a unit cube centered at a point in $X$ and thus
\[\mu_X(I'_{\ell})\geq C_R^{-1},\ \ \forall 1\leq \ell\leq L^d.\]
On the other hand, 
\[\bigcup_{\ell=1}^{L^d} I'_{\ell}=\left[-\frac12, L+\frac12\right]^d,\]
and each point in $\left[-1/2, L+1/2\right]^d$ can be covered by at most $2d+1$ of the cubes $I'_{\ell}$.
Therefore
\[C_R^{-1} L^d\leq \sum_{\ell=1}^{L^d}\mu_X(I'_{\ell})\leq (2d+1) \mu_X \Big(\left[-\frac12, L+\frac12\right]^d\Big)\leq (2d+1) C_R (L+1)^{\delta},\]
which contradicts \eqref{def:L}.
\end{proof}

Recall our definition of $\mathcal{C}_n$ and porosity in Definition \ref{def:porous}.
\begin{lemma}\label{lem:BDlemma210}
Let $X\subset [-1,1]^d$ be a $\delta$-regular set with constant $C_R$ on scales $\alpha_0$ to $\alpha_1$. 
Let $L$ satisfy \eqref{def:L}, and take $n\in \Z$ such that $\alpha_0\leq L^{-n-1}\leq L^{-n}\leq \alpha_1$. 
Then $X$ is porous at scale $L$ with depth $n$.
\end{lemma}

\section{Proof of Lemma \ref{lem:regularsetdamping}}\label{sec:app}
We follow the proofs of Theorem 3.2 and Lemma 4.1 in \cite{JZ}.
Let us start with introducing some notations.
\subsection{Hilbert transform}
Let $\mathcal{H}_0$ be the standard Hilbert transform defined as convolution with p.v.$\frac{1}{\pi x}$: For $f\in C_0^{\infty}(\R)$, (or more generally, $f\in L^1(\R, \la x \ra^{-1}\, dx)$)
\[\mathcal{H}_0(f)(x):=\frac{1}{\pi}\lim_{\varepsilon\to 0^+}\int_{|x-t|\geq \varepsilon}\frac{f(t)}{x-t}\, dt.\]
Let $\mathcal{H}$ be the modified Hilbert transform with the integral kernel that decays like $|x|^{-2}$ as $|x|\to \infty$:
\[\mathcal{H}(f)(x):=\frac{1}{\pi}\lim_{\varepsilon\to 0^+}\int_{|x-t|\geq \varepsilon}f(t)\Big(\frac{1}{x-t}+\frac{t}{t^2+1}\Big)\, dt, \quad f\in L^1(\R, \la x\ra^{-2}\, dx).\]
The advantage of $\mathcal{H}$ is that it applies to a larger space that contains $L^\infty(\R)$ as well as functions the grow like $|x|^{1-\epsilon}$ as $|x|\to \infty$.

If $f\in L^1(\R, \la x\ra^{-1}\, dx)$, then $\mathcal{H}(f)$ differs from $\mathcal{H}_0(f)$ by a constant.
Moreover, we have the inversion formula for all $f\in L^1(\R, \la x\ra^{-2}\, dx)$ with $\mathcal{H}(f)\in L^1(\R, \la x\ra^{-2}\, dx)$:
\EQ{\label{eq:Hinversion}
\mathcal{H}(\mathcal{H}(f))=-f+c(f),
}
where $c(f)$ is a real constant depending on $f$.

We will use the following example later in the proof.
\begin{example}\cite[Example 2.3]{JZ}
Let $f(x)=\log(x^2+1)$, then we can compute
\EQ{\label{eq:Hlogx^2+1'}
\mathcal{H}(f)'(x)=\mathcal{H}_0(f')(x)=-\frac{2}{x^2+1}.
}
\end{example}

\subsection{Hardy space and outer functions}
We recall the definition of Hardy space on the real line
\[H^2=H^2(\R)=\{ f\in L^2(\R): \supp\hat{f}\subset [0, \infty)\}.\]

If $f\in L^2(\R)$, then $f+i\mathcal{H}_0(f)\in H^2(\R)$.

The space of modulus of functions in $H^2$ can be characterized by the logarithmic integral: for $\omega\in L^2$, $\omega\geq 0$, we define
\[\mathcal{L}(\omega):=\int_{\R}\frac{\log{\omega(x)}}{1+x^2}\, dx.\]
\begin{theorem}\cite[Sec.1.5]{HJ}
If $f\in H^2$, and $\mathcal{L}(|f|)=-\infty$, then $f\equiv 0$. On the other hand, if $\omega\in L^2$, and $\mathcal{L}(\omega)>-\infty$, then there exists a function $f\in H^2$ with $|f|=\omega$, unique up to a multiplication by a complex constant with unit modulus.
\end{theorem}
If $\mathcal{L}(\omega)>-\infty$. Let $\Omega=-\log{\omega}$, then $\Omega\in L^1(\R, \la x\ra^{-2}\, dx)$. Therefore we can define 
$\tilde{\Omega}=\mathcal{H}(\Omega)$ and take
\EQ{\label{def:outer}
f=ae^{-\Omega+i\tilde{\Omega}},\ \ |a|=1.
}
We call functions of the form \eqref{def:outer} for general $\Omega\in L^1(\R, \la x\ra^{-2}\, dx)$ {\em outer functions}.
The class of outer functions is closed under multiplications. Moreover if two outer functions have the same modulus, then they differ by a complex constant with unit modulus.

The following lemma gives a sufficient condition of a function to be the modulus of the Fourier transform of a function supported in $[0,\sigma]$.
\begin{lemma}\label{lem:seven}\cite[Theorem 1]{seven}
Assume that $\omega=e^{-\Omega}\in L^2$ and $\mathcal{L}(\omega)>-\infty$. In addition, we assume that $\omega^2 e^{2\pi i \sigma x}$ is an outer function. Then there exsits $\psi\in L^2$ with $\supp \psi\subset [0,\sigma]$ and $|\what{\psi}|=\omega$.
\end{lemma}

\subsection{An effective multiplier theorem}
We prove an effective multiplier theorem. 
This proof is essentially in \cite[Section 3]{JZ}, the only change we make lies in the definition of $k(x)$ below.
Our modified definition makes sure that $k(x)$ is a constant function in a neighborhood of $0$, which leads to a pointwise lower bound of $\what{\psi}(x)$ on the whole interval $[-3/4,3/4]$.

\begin{theorem}\label{thm:multiplier}
Assume that $0<\omega\leq 1$ satisfies $\mathcal{L}(\omega)>-\infty$, and 
\[\|\mathcal{H}(\Omega)'\|_{L^\infty}\leq \frac{\pi}{2}\sigma,\]
where $0<\sigma<1/10$, $\Omega=-\log{\omega}$. Then there exists $\psi\in L^2(\R)$ with
\[\supp\psi \subset [0,\sigma],\ \ |\what{\psi}|\leq \omega,\]
and 
\[|\what{\psi}|\geq \frac{\sigma^{10}}{4\times 10^{11}} \omega,\ \ \text{on } [-3/4,3/4].\]
\end{theorem}

\begin{proof}
We first set 
\[\omega_0(x)=\frac{\omega(x)}{(x^2+T^2)^5},\ \ \Omega_0(x)=-\log{(\omega_0(x))},\]
with constant $T$ that will be specified later.
We then have
\[
\Omega_0=\Omega+5\log{(x^2+T^2)}.
\]

We compute
\EQ{\nn
\mathcal{H}(\log{(x^2+T^2)})(0)
=\lim_{\varepsilon\to 0^+}\int_{|t|\geq \varepsilon}  \log(t^2+T^2)\Big(\frac{1}{-t}-\frac{t}{t^2+1}\Big)\, d t,
}
in which the Integrand is an odd function. Hence the integration is zero. 
Therefore we have
\EQ{\label{eq:HOmega0=HOmega}
\mathcal{H}(\Omega_0)(0)=\mathcal{H}(\Omega)(0)+5\mathcal{H}(\log{(x^2+T^2)})=\mathcal{H}(\Omega)(0).
}

By \eqref{eq:Hlogx^2+1'}, we compute
\[
\mathcal{H}(\log(x^2+T^2))'=T^{-1}\mathcal{H}(\log(x^2+1))'(\cdot/T)=-\frac{2T}{x^2+T^2}.
\]
Thus if we choose $T=\frac{20}{\pi \sigma}\geq \frac{200}{\pi}\geq 60$, we have
\EQ{\label{eq:HOmega0'}
\|\mathcal{H}(\Omega_0)'\|_{L^\infty}\leq \|\mathcal{H}(\Omega)'\|_{L^\infty}+5\|\mathcal{H}(\log(x^2+T^2))'\|_{L^\infty}\leq \pi \sigma.
}

Let us define 
\[s_0(x)=\pi \sigma x+\mathcal{H}(\Omega_0)(x).\]
Hence by \eqref{eq:HOmega0=HOmega},
\[s_0(0)=\mathcal{H}(\Omega)(0),\]
depending only on $\omega$.

Let $s(x)$ be defined as follows
\[s(x)=s_0(x)-\pi k(x)-\frac{\pi}{2},\]
in which
\EQ{\label{def:k(x)}
k(x)=
\begin{cases}
\lfloor \frac{s_0(x)}{\pi} \rfloor,\ \ \ \ \ \quad \text{if } \frac{s_0(0)}{\pi} \in \left[\frac{1}{4}, \frac{3}{4}\right] \mod 1,\\
\\
\lfloor \frac{s_0(x)}{\pi} -\frac{1}{2} \rfloor,\ \ \text{if } \frac{s_0(0)}{\pi} \in \left[0, \frac{1}{4}\right) \bigcup  \left(\frac{3}{4}, 1\right) \mod 1,
\end{cases}
}
Note that our definition of $k(x)$ is different from that in \cite{JZ}. 
We modify the definition in order to make sure $k(x)$ is a constant near $x=0$. 
This will be explained and used later in the proof.

By \eqref{eq:HOmega0'}, $s_0(x)$ is a non-decreasing function and so is $k$.
Note also that by our definition of $s(x)$, we have
\EQ{\label{eq:sLinfty}
\|s\|_{L^\infty}\leq \pi.
}

Let $m=e^{-M}$ where $M=\mathcal{H}(s)$.
Next, we will estimate $M(x)=\mathcal{H}(s)(x)$.
We split the integral into three parts $M(x)=J_1(x)+J_2(x)+J_3(x)$, where
\EQ{\nn
J_1(x)&=\frac{1}{\pi}\int_{|x-t|<1/2} \frac{s(t)-s(x)}{x-t}\, dt;\\
J_2(x)&=\frac{1}{\pi}\int_{|x-t|<1/2} s(t)\frac{t}{t^2+1}\, dt;\\
J_3(x)&=\frac{1}{\pi}\int_{|x-t|\geq 1/2} s(t) \Big(\frac{1}{x-t}+\frac{t}{t^2+1}\Big)\, dt.
}
We estimate $J_2$ and $J_3$ in the same way as in \cite{JZ}. 
By \eqref{eq:sLinfty}, we have
\EQ{\label{eq:J2}
|J_2(x)|\leq \frac{1}{\pi}\cdot \|s\|_{L^\infty}\cdot \frac{1}{2}\leq \frac{1}{2}.
}
Also, we have
\EQ{\label{eq:J3}
|J_3(x)|\leq \frac{1}{\pi}\cdot \|s\|_{L^\infty} \int_{|x-t|\geq 1/2}\left| \frac{1}{x-t}+\frac{t}{t^2+1}\right|\, dt\leq 6\log{(|x|+2)}.
}

Finally, we need to bound $|J_1|$. By \eqref{eq:HOmega0'}, $s_0(x)=\pi \sigma x+\mathcal{H}(\Omega_0)(x)$ is non-decreasing with 
$\|s_0'\|_{L^\infty}\leq 2\pi \sigma$.
Since we assume $0<\sigma<1/10$, we have 
\[\|\pi^{-1} s_0'\|_{L^\infty}<\frac{1}{5}.\] 
This leads to the following
\begin{itemize}
\item if $\pi^{-1}s_0(0)\in [1/4, 3/4] \mod 1$, 
\[\frac{s_0(x)}{\pi}\in (0,1) \mod 1,\ \ \forall x\in \left[-\frac{5}{4}, \frac{5}{4}\right].\]
\item if $\pi^{-1}s_0(0)\in [0, 1/4)\cup (3/4,1) \mod 1$,
\[\frac{s_0(x)}{\pi}-\frac{1}{2} \in (0,1) \mod 1,\ \ \forall x\in \left[-\frac{5}{4}, \frac{5}{4}\right].\]
\end{itemize}
Recall our definition of $k(x)$ in \eqref{def:k(x)}, we know in each case $k(x)$ is a constant function on the interval $[-5/4, 5/4]$.

Thus for $x\in [-3/4,3/4]$, we have
\EQ{\label{eq:J10}
|J_1(x)|\leq \frac{1}{\pi}\int_{|x-t|<1/2}\left|\frac{s_0(t)-s_0(x)}{x-t}\right|\, dt\leq \frac{1}{\pi} \|s_0'\|_{L^\infty}\leq 2\sigma.
}
For all $x$, we only have a lower bound of $J_1$. 
Since $k$ is non-decreasing, we have
\EQ{\label{eq:J1all}
J_1(x)\geq \frac{1}{\pi}\int_{|x-t|<1/2}\frac{s_0(t)-s_0(x)}{x-t}\, dt\geq -2\sigma.
}

Now combining \eqref{eq:J2}, \eqref{eq:J3} with \eqref{eq:J10}, we have the following estimate of $M$ on $[-3/4,3/4]$:
\EQ{\label{eq:M0}
|M(x)|\leq 2\sigma+\frac{1}{2}+6\log{\frac{11}{4}}<7.
}
Using \eqref{eq:J1all} instead of \eqref{eq:J10}, we obtain that for all $x$,
\EQ{\label{eq:Mall}
M(x)\geq -2\sigma-\frac{1}{2}-6\log{(|x|+2)}>-1-6\log(|x|+2).
}

Next we will apply Lemma \ref{lem:seven} to $\tilde{\omega}=\frac{1}{3}m \omega_0$.
We check that $\tilde{\omega}$ satisfies all the assumptions.
First, by \eqref{eq:Mall}, we have
\[
0\leq \tilde{\omega}\leq \frac{e}{3}(|x|+2)^6 \omega_0\leq \frac{\omega}{x^2+T^2}.
\]
Hence $0\leq \tilde{\omega}\leq \omega$ and $\tilde{\omega}\in L^2$. Moreover
\[
\mathcal{L}(\tilde{\omega})=\mathcal{L}(m/3)+\mathcal{L}(\omega_0)>-\infty.
\]
By the construction $M=\mathcal{H}(s)$ and the inversion formula \eqref{eq:Hinversion}, we have
\[
\mathcal{H}(-2M-2\Omega_0)=2s-2\mathcal{H}(\Omega_0)-2c(M)=2\pi \sigma x-2\pi k(x)-\pi-2c(M),
\]
where $k(x)\in \Z$ and $c(M)$ is a real constant.
Therefore for some constant $a$ with $|a|=1$, we have
\[
\tilde{\omega}^2 e^{2\pi i \sigma x}=\frac{1}{9}e^{-2M-2\Omega_0+2\pi i \sigma x}=\frac{a}{9}e^{-2M-2\Omega_0+i\mathcal{H}(-2M-2\Omega_0)},
\]
which shows $\tilde{\omega}^2 e^{2\pi i \sigma x}$ is an outer function.

By Lemma \ref{lem:seven}, there exists $\psi\in L^2$ with $\supp(\psi)\subset [0,\sigma]$ and $|\what{\psi}|\leq \tilde{\omega}\leq \omega$. 
Furthermore, on $[-3/4,3/4]$, by \eqref{eq:M0}, and since $T=\frac{20}{\pi \sigma}$, we have
\[
|\what{\psi}(x)|=\tilde{\omega}(x)\geq \frac{1}{3}(1+T^2)^{-5} e^{-7} \omega(x)\geq \frac{\sigma^{10}}{4\times 10^{11}}\omega(x),
\]
as claimed.
\end{proof}

\subsection{Multiplier adapted to the regular sets}
Now we are in the place to finish the proof of Lemma \ref{lem:regularsetdamping}.
\begin{proof}
The proof is the same as Lemma 4.1 of \cite{JZ}, except that we will use Theorem \ref{thm:multiplier} instead of Theorem 3.2 of \cite{JZ}.
We briefly go through the various constants below.

It is shown in \cite{JZ} that one can construct a weight function $\omega$ such that
\EQ{\nn
&\omega(\xi)=\exp(-\la \xi\ra^{1/2})\geq 0.3,\ \ \forall \xi \in [-1,1],\\
&\omega(\xi)\leq \exp(-\la \xi\ra^{1/2}),\ \ \forall \xi\in \R,\\
&\omega(\xi)\leq \exp(-\Theta(|\xi |) |\xi|),\ \ \forall \xi \in Y,\  |\xi |\geq 10,\\
&\|\mathcal{H}(\omega)'\|_{L^\infty}\leq \frac{\iota^{-1} C_R^2}{\delta_1(1-\delta_1)},
}
where $0<\iota<1$ is an absolute constant.

Applying Theorem \ref{thm:multiplier} to $\omega^{c_3}$ with
\[\sigma=\frac{c_1}{5},\ \ c_3=\frac{\pi}{10}\iota\, c_1 C_R^{-2} \delta_1(1-\delta_1)<1.\]
We obtain $\psi$ with 
\EQ{\nn
&\supp\psi\subset \left[0, \frac{c_1}{5}\right],\\
&|\what{\psi}(\xi)|\geq \frac{c_1^{10}}{4 \times 10^{18}} \omega(\xi)^{c_3}\geq  \frac{3}{4\times 10^{19}} c_1^{10}, \ \ \forall \xi\in [-3/4,3/4],\\
&|\what{\psi}(\xi)|\leq \exp(-c_3 \la \xi\ra^{1/2}),\ \ \forall \xi \in \R,\\
&|\what{\psi}(\xi)|\leq \exp(-c_3 \Theta(|\xi |) |\xi |),\ \ \forall \xi \in Y,\ |\xi |\geq 10.
}
Finally, shifting $\psi$ by $c_1/10$ yields the desired function. 
\end{proof}

\end{document}